\documentclass[12pt, a4paper, oneside]{article}
\usepackage{amsmath, amsthm, amssymb, appendix, bm, graphicx, hyperref, mathrsfs,geometry,color,cite}
\usepackage{booktabs}
\usepackage{array}
\usepackage{caption}
\usepackage{subfigure}
\newtheorem{theorem}{Theorem}[section]
\numberwithin{equation}{section}
\newtheorem{definition}[theorem]{Definition}
\newtheorem{lemma}[theorem]{Lemma}
\newtheorem{corollary}[theorem]{Corollary}

\newtheorem{proposition}[theorem]{Proposition}
\newtheorem{remark}[theorem]{Remark}

\usepackage{sectsty}
\sectionfont{\fontsize{14}{16}\selectfont}

\usepackage{authblk}
\makeatletter
\renewcommand{\maketitle}{
	\begin{center}
		{\LARGE\textbf{ \@title} \par}
		\vskip 1em
		{ \@author \par}
		\vskip 1em
		{\small \@date \par}
	\end{center}
	\vskip 1em
	\@thanks
}
\makeatother

\title{Global Dynamics of Nonlocal Diffusion Systems on Time-Varying Domains}

\author[a,c]{Xiandong Lin}
\author[b,c]{Hailong Ye\thanks{Corresponding author: yhl@szu.edu.cn}}
\author[c]{Xiao-Qiang Zhao}

\affil[a]{School of Mathematics, Sun Yat-sen University, Guangzhou 510275, Guangdong, China}
\affil[b]{School of Mathematical Sciences, Shenzhen University, Shenzhen 518060, Guangdong, China}
\affil[c]{Department of Mathematics and Statistics, Memorial University of Newfoundland,
	St. John's, NL A1C 5S7, Canada}

\date{}
\begin{document}

\maketitle

\begin{abstract}

We propose a class of nonlocal diffusion systems on time-varying domains, and fully characterize their
asymptotic dynamics in the asymptotically fixed, time-periodic and unbounded cases. The kernel is not
necessarily symmetric or compactly supported, provoking anisotropic diffusion or convective effects.
Due to the nonlocal diffusion on time-varying domains in our systems, some significant challenges
arise, such as the lack of regularizing effects of the semigroup generated by the nonlocal operator,
as well as the time-dependent inherent coupling structure in kernel.
By investigating a general nonautonomous nonlocal diffusion system in the space of bounded and measurable functions,
we establish a comprehensive and unified framework to rigorously examine the threshold dynamics of the
original system on asymptotically fixed and time-periodic domains. In the case of an asymptotically unbounded
domain, we introduce a key auxiliary function to separate vanishing coefficients from nonlocal diffusions.
This enables us to construct appropriate sub-solutions and derive  the global threshold dynamics via the comparison principle.
The findings may be of independent interest and the developed  techniques, which do not rely on the
existence of the principal eigenvalue, are expected to find further
applications in the related nonlocal diffusion problems.
We also conduct numerical simulations based on a practical model to illustrate our analytical results.
\end{abstract}


\section{Introduction}

The overwhelming majority of studies devoted to diffusion processes assume they occur in a static media.
However, the expansion and contraction of spatial media is ubiquitous in reality.
In fact, the universe we live in undergoes an accelerating expansion \cite{Riess1998}, and elementary biological
processes such as morphogenesis (i.e., the process whereby organisms grow from a single cell)
involve tissue expansion \cite{Shraiman2005}. At the species scale, the habitats
of most organisms frequently shift in response to changes in ecological environments. For example, the depth
and surface area of many rivers and lakes fluctuate seasonally. During summer or the rainy season, rising water
levels expand habitats, benefiting organisms living within. Conversely, in winter or the dry season, the water
levels drop, reducing the available habitat. There are also several examples of habitats that are constantly
expanding, such as those of {\it Aedes} mosquitoes, which can transmit dengue fever, yellow fever, Zika virus,
and other infectious diseases. Global warming and frequent human activities have significantly expanded areas
suitable for their survival and reproduction (\cite{MR4251021,MR4784487}).
All these facts highlight the need to incorporate time-varying domains
into reaction-diffusion models to better capture population dynamics in changing environments.

The role of domain growth in reaction-diffusion models was first investigated by Newman and Frisch
\cite{Newman1979} in their studies of chick limb development. Kondo and Asai \cite{Kondo1995} demonstrated
that growth could account for the stripe patterning observed in the {\it Pomacanthus} fish, where additional
stripes emerge progressively as the fish matures (see also Painter et al. \cite{Painter-Maini1999}).
 Crampin et al.
\cite{Crampin-Hackborn-Maini,Crampin-Gaffney-Maini,Crampin-Maini2001,Crampin-Gaffney-Maini2002} further introduced
mathematical models to incorporate more general growth cases. Their results highlighted the critical role
of domain growth in enhancing pattern robustness and generating diverse biological patterns. In addition,
Chaplain et al. \cite{Chaplain-Ganesh-Graham2001} extended reaction-diffusion models to spherical surfaces,
demonstrating how domain geometry and growth drive spatial heterogeneity in solid tumor development.
Since then, population persistence and pattern formation on time-varying domains have attracted increasing
interest in the study of reaction-diffusion models, see, e.g.,  \cite{MR2643885,MR2110056,MR3711894,MR4251021,MR4636873}
and references therein. In particular, Lam et al. \cite{MR4784487} studied the asymptotic dynamics of
reaction-diffusion systems under the different evolving cases of domain using the theories of chain
transitive sets (\cite{MR3643081}) and principal eigenvalues.
In the aforementioned studies, all spatial diffusion processes in the models are described by Laplacian operator.

Recently, the nonlocal diffusion, describing the movements between both adjacent and nonadjacent spatial locations,
has been introduced in reaction-diffusion equations, see \cite{MR3990735,MR4312283,MR3000610,MR3957990,MR3968264,MR3401600}.
Besides, the associated nonlocal convolution-type operators have been widely employed to model various movements
arising in population dynamics, phase transition phenomena and image processing, see \cite{MR1811307,MR2028048,MR2480109}
and the references therein. Therefore, beyond the independent interest in specific models, a clearer and
deeper understanding of nonlocal diffusion is expected to offer valuable insights with implications across
multiple disciplines.

In light of these reasons and inspired by \cite{MR4784487,MR3957990}, we aim to propose a class of
reaction-diffusion systems with nonlocal diffusion and investigate their global dynamics on time-varying domains.
More precisely, we consider the following nonlocal diffusion system for $x\in \Omega_t$ and $t>0$,
\begin{equation}\label{intro-1}
\dfrac{\partial v}{\partial t}+(a\cdot \nabla) v+(\nabla\cdot a)v=D\int_{\mathbb{R}^n}J(x-z)v(t, z)\,\mathrm{d}z-Dv+f(v),
\end{equation}
under the different evolving conditions of domain $\Omega_t$, which incorporates the homogeneous Dirichlet boundary condition
\begin{equation}\label{intro-2}
	v(t, x)=0,\quad t>0, \   x\in\mathbb{R}^n\backslash \Omega_t,
\end{equation}
and the initial condition
\begin{equation}\label{intro-3}
	v(0, x)=v_0(x), \quad x\in\Omega_0.
\end{equation}
For the detailed derivation of \eqref{intro-1}--\eqref{intro-3}, we refer to Section 2.
Here, $v=(v_{1},\dots,v_{m})^{T}$ represents the vector of densities of $m$ interacting species, $m$ and $n$ are positive
integers, $D={\rm diag}\left\lbrace d_{1},\dots, d_{m}\right\rbrace$, $d_{i}>0$ is diffusion coefficient for $i=1,\dots,m$,
$f(v)=(f_1(v),\dots,f_m(v))^T$ denotes the reaction terms, the kernel $J$ satisfies the assumption {\bf (J)} below, and
$\Omega_{t}\subset \mathbb{R}^{n}$ is simply connected and bounded with smooth boundary $\partial\Omega_t$ for all $t\ge0$.
The condition \eqref{intro-2} indicates that the habitat outside $\Omega_t$ is so hostile that the species die immediately
upon entering (\cite{MR2028048}).

Notice that for each $i=1,\dots,m$, the diffusion of the density $v_i$ at a point $x$ and time $t$ depends on the values
of $v_i$ at all points in the set $\{x\}+{\rm supp}J$, which is what makes the diffusion operator nonlocal in system \eqref{intro-1}.
As stated in \cite{MR1999098}, the kernel $J(x-z)$ denotes the probability distribution of an individual jumping
from location $z$ to location $x$, then the rate at which individuals are arriving to position $x$ from all other places
is given by $d_i\int_{\mathbb{R}^n}J(x-z)v_i(t, z)\,\mathrm{d}z$, while the rate at which they are leaving location $x$
to travel to all other sites is given by $-d_i\int_{\mathbb{R}^n}J(z-x)v_i(t, x)\,\mathrm{d}z=-d_iv_i(t,x)$.

In addition to the nonlocal feature, system \eqref{intro-1} incorporates two additional terms associated with the
flow $a$, which arises from domain evolution: $(a\cdot \nabla) v$ and $(\nabla\cdot a)v$.
The former represents material transport around the domain at a rate determined by the flow, while the latter accounts
for dilution or concentration due to local volume changes. As the first step in the analysis, it is essential
to reformulate the model \eqref{intro-1}--\eqref{intro-3} on a fixed domain. For this purpose, we focus on a class of
time-varying domains characterized by linear isotropic deformation with spatially uniform rates.
Then a smooth positive function $\rho$, with $\rho(0)=1$, is introduced to measure the growth of the domain:
$\Omega_t=\rho(t)\Omega_0$ for $t\ge0$.
Detailed discussions and fundamental assumptions are provided in Section 2.
Thus, by denoting $u(t,y):=v(t,\rho(t)y)$, we derive the following equivalent model,
\begin{eqnarray}\label{intro-4}
\begin{cases}
\dfrac{\partial u}{\partial t}=D\displaystyle\int_{\Omega_0}J_{\frac{1}{\rho(t)}}(y-z)u(t,z)\,\mathrm{d}z-Du
-\dfrac{n\dot{\rho}(t)}{\rho(t)}u+f(u), &\, t>0, \  y\in\Omega_0,
			\\[3mm]
u(0,y)=u_0(y),& \, y\in\Omega_0,
\end{cases}
\end{eqnarray}
where
\begin{equation}\label{J-rho}
J_{\frac{1}{\rho(t)}}(y):=\rho^n(t)J(\rho(t)y),\quad t\ge0.
\end{equation}
Obviously, the transport effect is excluded in \eqref{intro-4} at the cost of the kernel being inherently coupled
with $\rho(t)$, making them inseparable. In view of the equivalence between \eqref{intro-1}--\eqref{intro-3} defined
on $\Omega_t$ and \eqref{intro-4} reformulated on $\Omega_0$, we will focus on the global dynamics of \eqref{intro-4}
under the following three asymptotic cases of $\rho$. Each case corresponds to a specific type of domain evolution.
Notably, the first two cases describe asymptotically bounded domains since $\overline{\lim\limits_{t\to\infty}}\rho(t)<\infty$.
\begin{itemize}
\item  {\it Asymptotically fixed case:}
$$
\lim\limits_{t\to+ \infty} \rho (t)= \rho_\infty>0,\quad \lim\limits_{t\to+ \infty} \dot\rho(t)=0.
$$
In this case, the domain $\Omega_t$ asymptotically approaches a fixed domain as $t\to+\infty$. A biologically reasonable example is
logistic or saturated growth (\cite{MR2643885,MR2110056}) of the form $\rho(t)=\frac{\rho_\infty e^{\kappa t}}{\rho_\infty-1+e^{\kappa t}}$,
where $\kappa>0$. This can be seen easily by noting
$\frac{\mathrm{d} \rho}{\mathrm{d} t}=\kappa\rho\left(1-\frac{\rho}{\rho_{\infty}}\right)$.
\item  {\it Asymptotically time-periodic case:}
$$
\lim\limits_{t\to +\infty}  (\rho (t) -\rho_T(t))= \lim\limits_{t\to +\infty} (\dot\rho (t)-\dot\rho_T(t))=0.
$$
Here, $\rho_T$ denotes a positive $T$-periodic function for some $T>0$, and the domain $\Omega_t$ asymptotically exhibits periodic oscillations
as $t\to+\infty$ (\cite{MR4784487,MR4636873}).
\item  {\it Asymptotically unbounded case:}
$$
\lim\limits_{t\to +\infty} \rho (t)= +\infty,\quad \lim\limits_{t\to +\infty}\frac{\dot\rho(t)}{\rho (t)} =k\ge0.
$$
In this case, the domain $\Omega_t$ eventually expands to $\mathbb{R}^n$ as $t\to+\infty$, with the relative growth
rate of $\rho(t)$ approaching a nonnegative constant $k$. Biologically relevant examples of $\rho$ include exponential
growth $\rho(t)=e^{kt}$ and a simpler alternative $\rho(t)=1+\kappa t$ with $\kappa>0$, both of which have been discussed in \cite{MR2643885,MR2110056}.
\end{itemize}

It should be noted that, though the problem \eqref{intro-4} can be regarded as a nonlocal counterpart of those studied in \cite{Crampin-Hackborn-Maini,Crampin-Gaffney-Maini,Crampin-Maini2001,Crampin-Gaffney-Maini2002,MR4784487,MR2643885,MR2110056,MR3711894,MR4251021},
it presents greater challenges for the analysis of its global dynamics. For example, there is no regularizing effect in general due
to nonlocal diffusion, which renders the arguments employed in the local setting invalid, including the theories of asymptotically autonomous semiflows
and chain transitive sets \cite{MR3643081}. Furthermore, in the local version of \eqref{intro-4} as shown in
\cite{MR4784487,MR3711894,Crampin-Gaffney-Maini,MR2643885}, the diffusion term is given by $\frac{d}{\rho^2(t)}\Delta u$.
This decoupling of the Laplacian operator from $\rho(t)$ allows it to be treated independently, enabling the principal
eigenvalue theory as a powerful tool for analyzing the asymptotic dynamics of local systems.
However, this method is no longer applicable to
the nonlocal system \eqref{intro-4} because of the inherent coupling structure in \eqref{J-rho}, particularly in the asymptotically
unbounded case, where the diffusion effect vanishes as $\rho(t)\to+\infty$.
In addition, the principal eigenvalue of nonlocal diffusion operators do not exist in general.
Therefore, the previous arguments need to be improved or reconceived.

In Section 3, to characterize the asymptotic dynamics of \eqref{intro-4} in the first two cases, we consider a
general nonautonomous nonlocal diffusion system given by
\begin{equation}\label{intro-5}
\dfrac{\partial u}{\partial t}=D\displaystyle\int_{\Omega_{0}}\mathcal{K}(t,x-y)u(t,y)\,\mathrm{d}y-Du +g(t,x,u),\quad  t>0, x\in \Omega_{0}.
\end{equation}
Let $\omega(\Phi)$ denote the spectral bound of the evolution family generated by the corresponding linearized system.
When $\mathcal{K}(t,x)=\mathcal{K}(x)$ and $m=1$,
the global dynamics of \eqref{intro-5} in the case $\omega(\Phi)>0$ was established by Rawal and Shen \cite{MR3000610}, where the
construction of a perturbed problem admitting a principal eigenvalue is a crucial step in proving the existence
and uniqueness of a positive time-periodic solution. Subsequently, Shen and Vo \cite{MR3957990} extended the above results to the case
$\omega(\Phi)<0$. It is noteworthy that, in the critical case $\omega(\Phi)=0$, the uniqueness of a nonnegative time-periodic solution
and the global dynamics of system \eqref{intro-5} remain an interesting and challenging open problem due to several overwhelming difficulties,
including the absence of regularizing effects and the principal eigenvalue for the nonlocal diffusion operator.
If $g(t,x,u)=g(x,u)$, the global dynamics in the critical case was established in \cite{Berestycki-Coville-VoJFA,Berestycki-Coville-VoJMB} using a
Harnack-type inequality for nonlocal elliptic equations and bootstrap arguments. For traveling waves and spreading speeds in
nonlocal diffusion equations with time heterogeneities in both the kernel and the reaction, we refer to the recent works  \cite{MR4454378,MR4725136}.

Here we propose some novel approaches for rigorously analyzing the threshold dynamics of \eqref{intro-5} in the time-periodic case based on $\omega(\Phi)$.
The techniques developed here are independent of principal eigenvalue and may be applicable to a wide
class of related nonlocal diffusion problems.
Specifically, we establish the comparison principle and well-posedness of \eqref{intro-5} in $C( \mathbb{R}_{+},\hat X_{+})$,
and develop a deeper understanding of time-periodic solutions in $\tilde X_{+}^{v}$,  where $\hat X$ and $\tilde X$ are two spaces of  bounded and measurable functions  on
$\overline\Omega_{0}$ and $\mathbb{R}\times\overline\Omega_{0}$, respectively, see Lemma \ref{le2.1} and Theorem \ref{generalized-dynamical} for more details.
These results will be utilized in Section 4 to analyze the threshold dynamics of \eqref{intro-4} in the asymptotically
fixed and time-periodic cases.

In Section 5, we consider the asymptotically unbounded case. To address the challenges posed by the inherent coupling structure
in \eqref{J-rho}, we construct a nonnegative function $\phi\in C^1(\mathbb{R}^n)$, supported on $\Omega_0$, satisfying
\begin{equation*}
		\int_{\Omega_0}J_{\frac{1}{\rho(t)}}(x-y)\phi(y)\,\mathrm{d}y-\phi(x)\ge -\dfrac{1}{\rho(t)}\phi(x), \quad x\in\mathbb{R}^n, t\ge T.
	\end{equation*}
This inequality partially decouples the vanishing factor $\frac{1}{\rho(t)}$ from the effects of nonlocal diffusion.
Combined with the entire solution of limiting system corresponding to \eqref{intro-4}, this auxiliary function enables
the construction of appropriate subsolutions, leading to the derivation of global threshold dynamics via the comparison principle.

For a given $v\in\mathrm{Int}(\mathbb{R}^{m}_{+})$, let $\mathbb{M}:= \left\{u\in\mathbb{R}^{m}: 0\le u \le v\right\}$.
Through out the paper, we always assume that $J$ and $f$ satisfy the following assumptions:
\begin{itemize}
	\item[\textbf{(J)}] $ J: \mathbb{R}^n\to\mathbb{R}$ is a continuous nonnegative function with $ J(0)>0$ and $\int_{\mathbb{R}^{n}}J(x)\,\mathrm{d}x=1$;
	\item[\textbf{(F)}] $f\in C^{0,0,1}(\mathbb{R}\times\overline\Omega_{0}\times\mathbb{R}^{m},\mathbb{R}^{m})$; $\mathcal{D}f(u):=(\frac{\partial f_{i}(u) }{\partial u_{j}})_{m\times m}$ is cooperative and irreducible for  $ u\in \mathbb{M}$;  $f(0)=0 $, and there exists $\sigma>0$ such that for each $0\le i\le m$,  $f_{i}(u)-n \frac{\dot\rho(t)}{\rho(t)}v_{i}\le -\sigma $ whenever $t\ge 0$ and  $u\in \mathbb{M}$ with $u_{i}=v_{i}$;   $\alpha f(u)< f(\alpha u)$ for all  $u\in \mathbb{M}$ with $u\in\mathrm{Int}(\mathbb{R}^{m}_{+})$ and $\alpha\in (0,1)$.
\end{itemize}

We remark that the kernel $J$ is not necessarily symmetric or compactly supported, which means that individuals have greater
probability of jumping in one direction than in others, provoking anisotropic diffusion or convective effects
(\cite{MR2722295,MR2356418}).
In contrast to the previous studies \cite{MR4784487,MR2110056,MR4251021,Crampin-Maini2001}, we remove the assumption that $\dot{\rho} > 0$, allowing the domain
to expand and shrink, as described in \cite{MR4636873} for positive and negative growth.

The remainder of this paper is organized as follows. In Section 2, we derive the nonlocal diffusion model on a
time-varying domain and reformulate it on a fixed domain. In Section 3, we provide a comprehensive framework to
characterize the threshold dynamics of a general nonautonomous nonlocal diffusion system.
These rusults subsequently are utilized in Section 4 to rigorously examine the global dynamics of the original
nonlocal system on asymptotically fixed and time-periodic domains, respectively.
In Section 5, we investigate the global dynamics on asymptotically unbounded domains.
Finally, numerical simulations are given in Section 6.

\section{Nonlocal model: a derivation}

In this section, we derive the nonlocal diffusion systems formulated on a time-varying domain.
To this end, some basic assumptions are required on the time-varying domain.
Following the approach of \cite{Crampin-Hackborn-Maini,Crampin-Gaffney-Maini,Crampin-Maini2001,Crampin-Gaffney-Maini2002},
we adopt a general framework without distinguishing between specific tissue or media types, enabling the inclusion
of properties for any given tissue or medium. As a result, no constitutive equations are introduced, and the tissue
is assumed to be incompressible, that is, the domain undergoes deformation and expansion due to growth, with no
accompanying change in density.

\subsection{Time-varying domain and flow}

Suppose that tissue which occupies the domain $\Omega_0\subseteq\mathbb{R}^n$ at time $t=0$ deforms and expands
so that at a subsequent time $t$ it occupies a new continuous, simply connected and bounded domain $\Omega_t$,
and that each point of $\Omega_t$ can be identified with position vector $x$, which is occupied at $t$ by
the particle in tissue with position vector $\mathcal{X}$ at the time $t=0$. Then the motion of tissue can be
described by specifying the dependence of the positions $x$ of the particles of tissue at time $t$ on their
initial positions $\mathcal{X}$ at time $t=0$, that is,
$$
x=\Gamma(t, \mathcal{X}), \quad t\ge0,
$$
for all $\mathcal{X}\in\Omega_0$, where $\Gamma: \mathbb{R}_+\times\Omega_0\to\mathbb{R}^n$ is a smooth injection.

In the spatial description, the flow $a(t,x)$ represents the velocity at time $t$ of the particle which was at
$\mathcal{X}=\Gamma^{-1}(t, x)$ initially, that is,
\begin{equation}\label{flow-velocity}
	a(t, x)=\dfrac{\partial\Gamma(t, \mathcal{X})}{\partial t}\Big|_{\mathcal{X}=\Gamma^{-1}(t,x)},
\end{equation}
which may be typically generated by rigid motions and deformations of the tissue. Here, we neglect rigid motions
since they do not affect the reaction and diffusion processes, and assume that the flow arises
solely from the (positive or negative) growth of the tissue.
It is worth noting that the velocity gradient tensor $\nabla a$, which describes the relative velocity of each
particle with respect to its neighbors, is locally determined by some constitutive equations characterizing
tissue properties, such as prepatterns in growth factors, as well as cellular or sub-cellular structures influencing
the direction of growth. These detailed factors, however, are beyond the scope of this paper.
By continuum mechanics \cite{MR636255}, the diagonal components of $\nabla a$ represent the rate of extension.
Consequently, for any fixed time $t\ge0$, the domain $\Omega_t$ expands if $\nabla\cdot a(t,x)>0$ for all $x\in\Omega_t$,
and contracts if $\nabla\cdot a(t,x)<0$ for all $x\in\Omega_t$.

\subsection{Nonlocal diffusion system}
Let $v(t,x)=(v_1(t,x), v_2(t,x),\dots,v_m(t,x))^T$ denote the vector of densities of $m$ interacting species at
position $x\in\Omega_t$ and time $t$. Inspired by the idea in \cite{MR3023366}, we define the nonlocal flux functional $\mathcal{F}(\cdot,\cdot\,; \psi)$
for any pair of measurable domains $\Omega'$ and $\Omega''$, where $\psi$ is a nonnegative integrable function defined on
$\Omega'\cup\Omega''$, as follows:
$$
\mathcal{F}\big(\Omega', \Omega''; \psi\big)=\int_{\Omega''}\int_{\Omega'}J(x-z)\psi(z)\,\mathrm{d}z\mathrm{d}x-\int_{\Omega'}\int_{\Omega''}J(z-x)\psi(x)\,\mathrm{d}x\mathrm{d}z,
$$
which quantifies the net flux from $\Omega'$ to $\Omega''$ arising from the nonlocal interactions of $\psi$ over these domains.
Obviously, $\mathcal{F}\big(\Omega', \Omega''; \psi\big)=-\mathcal{F}\big(\Omega'', \Omega'; \psi\big)$.
Then, the nonlocal balance law \cite{MR3010838} for $v$ over an elemental volume $V_t$ evolving in time is given by
$$
\dfrac{\mathrm{d}}{\mathrm{d}t}\int_{V_t}v(t, x)\,\mathrm{d}x=D\mathcal{F}\big(\mathbb{R}^n\backslash V_t, V_t; v\big)+\int_{V_t}f(v)\,\mathrm{d}x,
$$
where $D$ represents the diffusion coefficient, and $f(v)$ denotes the reaction term.
Applying the Reynolds transport theorem for the left-hand side, we obtain
$$
\int_{V_t}\left(\dfrac{\partial v}{\partial t}+(a\cdot \nabla) v+(\nabla\cdot a)v\right)\mathrm{d}x
=D\mathcal{F}\big(\mathbb{R}^n\backslash V_t, V_t; v\big)+\int_{V_t}f(v)\,\mathrm{d}x.
$$
The domain is fixed in time and so we may differentiate through the integral.
By the arbitrariness of $V_t$, the governing equation becomes
\begin{equation}\label{general_model-1}
\dfrac{\partial v}{\partial t}+(a\cdot \nabla) v+(\nabla\cdot a)v=D\int_{\mathbb{R}^n}J(x-z)v(t, z)\,\mathrm{d}z-Dv+f(v)
\end{equation}
for $x\in\Omega_t,\,t>0$.

Compared to the standard reaction-diffusion systems, the time-varying domain introduces two additional terms
into \eqref{general_model-1}. The first term, $(a\cdot \nabla) v$, represents the transport of material around
the domain driven by the flow. The second term, $(\nabla\cdot a)v$, accounts for dilution due to local volume
expansion when $\nabla\cdot a>0$, or concentration due to local volume contraction when $\nabla\cdot a<0$.
Consequently, by incorporating the homogeneous Dirichlet boundary conditions
\begin{equation}\label{general_model-2}
	v(t, x)=0,\quad x\in\mathbb{R}^n\backslash \Omega_t, \ t>0,
\end{equation}
and the initial condition
\begin{equation}\label{general_model-3}
	v(0, x)=v_0(x), \quad x\in\Omega_0,
\end{equation}
we formulate the initial-boundary value problem \eqref{general_model-1}--\eqref{general_model-3} with nonlocal
diffusion on the time-varying domain.

\subsection{Linear isotropic deformation}

In most cases, the properties of solutions to \eqref{general_model-1}--\eqref{general_model-3} on time-varying
domains are difficult to study, even after reformulating the problem on a fixed domain, where the resulting form
is significantly more complex, see for example \cite{MR2643885}.
In this paper, we focus on a special class of time-varying domains undergoing linear isotropic deformation,
with deformation rates independent of spatial position.
By neglecting the rigid-body translations and rotations of the tissue, the coordinate system can be appropriately
chosen such that there is a reference point which remains at the origin of the coordinate all the time. Then
there exists a smooth positive function $\rho$ defined on $[0,+\infty)$ with $\rho(0)=1$ such that
$$
x=\Gamma(t, \mathcal{X})=\rho(t)\mathcal{X},\quad \mathcal{X}\in\Omega_0,\ t\ge0.
$$
The flow can then be determined by \eqref{flow-velocity} as $a(t, x)=\dot{\rho}(t)\mathcal{X}=x\dot{\rho}(t)/\rho(t)$.
To facilitate analysis, we transform the spatial coordinates into the reference domain $\Omega_0$, which is
independent of time, and simplify notation by using the coordinate $y \in \Omega_0$ instead of $\mathcal{X}$ hereafter.
Define
$$
u(t,y)=v(t, x)=v(t,\rho(t)y), \quad y\in\Omega_0,\ t>0.
$$
Then transformed nonlocal system of \eqref{general_model-1}--\eqref{general_model-3}
with nonautonomous coefficients becomes
\begin{eqnarray}\label{fixed-domain-model}
	\begin{cases}
		\dfrac{\partial u}{\partial t}=D\displaystyle\int_{\Omega_0}J_{\frac{1}{\rho(t)}}(y-z)u(t,z)\,\mathrm{d}z-Du-\dfrac{n\dot{\rho}(t)}{\rho(t)}u+f(u), &\, y\in\Omega_0,\, t>0,
			\\[3mm]
		u(0,y)=u_0(y),& \, y\in\Omega_0,
	\end{cases}
\end{eqnarray}
where $J_{\frac{1}{\rho(t)}}(y)=\rho^n(t)J(\rho(t)y)$ satisfying
$\int_{\mathbb{R}^{n}}J_{\frac{1}{\rho(t)}}(y)\,\mathrm{d}y=1$ for all $t\ge0$.

\begin{remark}
Assume that $J$ is a smooth, nonnegative, and symmetric function supported on the unit ball in $ \mathbb{R}^n$,
satisfying $\int_{\mathbb{R}^n}J(y)\mathrm{d}y=1$. By appropriately rescaling the kernel $J_{\frac{1}{\rho(t)}}$
in \eqref{fixed-domain-model} and taking the limit as the scaling parameter goes to zero, the solutions of
the nonlocal system \eqref{fixed-domain-model} converge uniformly to the solution of the local analogous system
(see \cite{MR4784487}) on $[0,T]$ for any fixed $T>0$. For the proof, please refer to \cite{MR3401600}. If $J$
is asymmetric, similar convergence results can be found in \cite{MR2722295}.
\end{remark}

\section{A general nonautonomous nonlocal diffusion system}

In this section, we consider a class of nonautonomous nonlocal diffusion systems within a fixed domain:
\begin{equation}\label{eq2.2}
\left\{
{\begin{array}{*{20}{l}}
		\dfrac{\partial u}{\partial t}=D\displaystyle\int_{\Omega_{0}}\mathcal{K}(t,x-y)u(t,y)\,\mathrm{d}y-Du +g(t,x,u), & t>0, x\in \Omega_{0}, \\[3mm]
		u(0,x)=u_{0}(x),  & x \in \Omega_{0}.\\
\end{array}} \right.
\end{equation}
where $\mathcal{K}(t,x-y)u(t,y):=(\mathcal{K}_{1}(t,x-y)u_{1}(t,y),\dots,\mathcal{K}_{m}(t,x-y)u_{m}(t,y))^{T} $.
We first prove the comparison principle and establish the well-posedness of system \eqref{eq2.2}, and then
provide a comprehensive framework to characterize the threshold dynamics of \eqref{eq2.2} in the time-periodic case based
on the spectral bound. These results will  be utilized in Section 4 to rigorously examine the global dynamics of
the original system on asymptotically fixed and time-periodic domains, respectively.

\subsection{Comparison principle and well-posedness}
The following assumptions are imposed on $\mathcal{K}$ and $g$:
\begin{itemize}

\item[\textbf{(K)}] For each $1\le i\le m$, $\mathcal{K}_{i}$ is a nonnegative continuous function with
	$\mathcal{K}_{i}(t,0)>0$, and $\int_{\mathbb{R}^{n}}\mathcal{K}_{i}(t,x)\mathrm{d}x=1$
	for all $ t\in \mathbb{R}$;

\item[\textbf{(G1)}] $g\in C^{0,0,1}(\mathbb{R}\times\overline\Omega_{0}\times\mathbb{R}^{m},\mathbb{R}^{m})$;
    $(\frac{\partial g_{i}(t,x,u) }{\partial u_{j}})_{m\times m}$ is cooperative for all
	$x\in\overline\Omega_{0}, t\ge 0, 0\le u\le v$,	and for each $1\le i\le m$,
	$g_{i}(t,x,0)=0, g_{i}(t,x,u)\le 0 $ whenever $u\in \mathbb{M}$	with $u_{i}=v_{i}$;

\item[\textbf{(G2)}]  $(\frac{\partial g_{i}(t,x,u) }{\partial u_{j}})_{m\times m}$ is irreducible for all $t\ge0$,
    $x\in \overline\Omega_{0}$ and $u\in  \mathbb{M}$.

\end{itemize}
Define the norm on $ \mathbb{R}^{m}$ as $\left\| a\right\|_{\mathbb{R}^{m}}:= \max\limits_{i=1,\dots,m}\left| a_{i}\right|$,
where $a=(a_{1},\dots,a_{m})^{T}\in\mathbb{R}^{m} $. We denote $X:=C(\overline\Omega_{0},\mathbb{R}^{m})$,
equipped with the maximum norm
$\left\| \varphi\right\| _{X}:=\max\limits_{x\in\overline\Omega_{0} } \left\| \varphi(x)\right\|_{\mathbb{R}^{m}}$,
and the positive cone $X_{+}:=C(\overline\Omega_{0},\mathbb{R}^{m}_{+})$.
Inspired by \cite{MR2556498,MR3968264}, we introduce
$$
\hat X:=\{\varphi:\overline\Omega_{0}\to\mathbb{R}^{m}; \varphi \textrm{ is bounded and Lebesgue measurable on } \overline\Omega_{0}\}
$$
which is equipped with the norm
$\left\|\varphi\right\| _{\hat X}:=\sup\limits_{x\in\overline\Omega_{0} } \left\|\varphi(x)\right\|_{\mathbb{R}^{m}}$,
and the corresponding positive cone
$\hat X_{+}:=\left\lbrace\varphi\in \hat X: \varphi(x)\ge 0,\,x\in\overline\Omega_{0}\right\rbrace$.
It can be readily verified that $(\hat X, \|\cdot\|_{\hat X}) $ is a Banach space and $\mathrm{Int}(\hat X_{+})\neq\emptyset$.

Let $a,b\in\mathbb{R}^{m} $. We write $a-b\ge 0$ if $a-b\in\mathbb{R}^{m}_{+}$; $a-b> 0$ if
$a-b\in\mathbb{R}^{m}_{+}\setminus\left\{0\right\}$; and $a-b\gg 0$ if $a-b\in\mathrm{Int}(\mathbb{R}^{m}_{+})$.
Similar to the order defined in $\mathbb{R}^{m}$, we can define the partial order induced by the positive cone
in $X$ and $ \hat X$, respectively. For any given $a\in \mathbb{R}^{m}$, we denote
$$
X_{+}^{a}:=\left\lbrace\varphi\in X: 0\le\varphi(x)\le a,\,x\in\overline\Omega_{0}\right\rbrace \ \textrm{ and } \ \hat X_{+}^{a}:=\left\lbrace
\varphi\in\hat X: 0\le\varphi(x)\le a,\,x\in\overline\Omega_{0}\right\rbrace.
$$

\begin{definition}
Let $l\in(0,+\infty]$. The function $u\in C(\left[0, l\right) , \hat X)$ is called a supersolution (or subsolution)
of the equations corresponding to \eqref{eq2.2} if $\frac{\partial u}{\partial t}(t,x)$ exists for $t\in[0,l)$ and
$x\in \overline \Omega_{0}$, and satisfies
\begin{equation*}
\dfrac{\partial u}{\partial t} \ge (\le)\ D\displaystyle\int_{\Omega_{0}}\mathcal{K}(t,x-y)u(t,y)\,\mathrm{d}y-Du +g(t,x,u), \quad t\in(0,l), x\in \Omega_{0}.
\end{equation*}

The function $u\in C(\left[0, l\right) , \hat X)$ is called a solution of system  $\eqref{eq2.2}$ if it is both
a supersolution and subsolution, and $u(0,x)=u_{0}(x)$.
\end{definition}

We first establish the maximum principle for the following linear system, and then apply it to derive the
comparison principle for the problem \eqref{eq2.2},
\begin{equation}\label{linear-system}
\dfrac{\partial u}{\partial t}=  D\displaystyle\int_{\Omega_{0}}\mathcal{K}(t,x-y)u(t,y)\,\mathrm{d}y-Du(t,x) +M(t,x)u(t,x)
\end{equation}
for $t>0$, $x\in\Omega_{0}$, where $M(t,x)=(m_{ij}(t,x))_{m\times m}$  is cooperative for all $x\in\overline\Omega_{0}, t\ge 0$,
with $m_{ij}\in C(\mathbb{R}_{+}\times\overline\Omega_{0},\mathbb{R})$, and there exists a point $x_{0}$ such that
$M(t,x_{0})$ is irreducible for all $t\ge0$. For this purpose, it is necessary to introduce an auxiliary lemma.
Define a set
$$
\hat {\mathbf 0}:=\{\phi\in \hat X: \phi(x)=0\textrm{ almost everywhere in } \overline\Omega_{0}\}.
$$

\begin{lemma}\label{le2.2-1}
Assume that {\rm\textbf{(K)}} holds and $T_0>0$. Let $u\in C([0,T_{0}],\hat X)$ and satisfy the system \eqref{linear-system}
on $[0,T_0]\times\overline\Omega_{0}$. If $u(0,\cdot)\in\hat {\mathbf 0}$, then $u(t,\cdot)\in\hat {\mathbf 0}$
for any $t\in(0, T_0]$.
\end{lemma}

\begin{proof}
Let $u:=(u_1,\dots,u_m)^T$ and define the vector-valued sign function ${\rm sgn}(u):=({\rm sgn}(u_1),\dots,{\rm sgn}(u_m))^T$.
Denote $\tilde{u}:=\sum_{i=1}^m|u_i|$. Taking the inner product of the governing equation for $u$ with ${\rm sgn}(u)$,
and integrating the resulting equation over $\Omega_0$, we obtain
\begin{align*}
& \frac{\mathrm{d}}{\mathrm{d}t}\int_{\Omega_0}\tilde{u}(t,x)\,\mathrm{d}x
\\[2mm]
& \le \sum_{i=1}^md_i\left( \int_{\Omega_{0}} \int_{\Omega_{0}}\mathcal{K}_i(t,x-y)\left| u_i(t,y)\right| \,\mathrm{d}y\mathrm{d}x
    -\int_{\Omega_{0}}  u_i{\rm sgn}(u_i)\,\mathrm{d}x\right)+C\int_{\Omega_0}\tilde{u}\,\mathrm{d}x
\\[2mm]
& \le \sum_{i=1}^md_i\left( \int_{\Omega_{0}}|u_i(t,y)|\,\mathrm{d}y- \int_{\Omega_{0}}|u_i(t,x)|\,\mathrm{d}x\right)
    +C\int_{\Omega_0}\tilde{u}\,\mathrm{d}x
\\[2mm]
&  =  C\int_{\Omega_0}\tilde{u}(t,x)\,\mathrm{d}x,
\end{align*}
where $C>0$. By the Gronwall inequality, it follows that $\int_{\Omega_0}\tilde{u}(t,x)\,\mathrm{d}x=0$ for $t\in (0,T_{0}]$.
This completes the proof.
\end{proof}

\begin{lemma}[Maximum principle]\label{le2.2}
Assume that {\rm\textbf{(K)}} holds and $T_0>0$. Let $u\in C([0,T_{0}],\hat X)$ satisfy that $\frac{\partial u}{\partial t}(t,x)$
exists for $t\in[0,l)$ and $x\in \overline \Omega_{0}$. Moreover, $u$ satisfies the differential inequality:
$$
\dfrac{\partial u}{\partial t}\ge  D\displaystyle\int_{\Omega_{0}}\mathcal{K}(t,x-y)u(t,y)\,\mathrm{d}y-Du(t,x) +M(t,x)u(t,x),\quad  t\in (0,T_{0}], x\in \overline{\Omega}_{0}.
$$
\begin{itemize}
  \item [(1)] If $u(0,\cdot)\in\hat X_{+}$, then $u(t,\cdot)\in\hat X_{+}$ for any $t\in\left(0,T_{0}\right]$.
  \item [(2)] If $u(0,\cdot)\in\hat X_{+}\setminus\hat{\mathbf 0}$, then $u(t,\cdot)\in \mathrm{Int} (\hat X_{+})$ for any $t\in\left(0,T_{0}\right]$.
  \item [(3)] If $u(0,x)\ge0$ a.e. in $\Omega_0$, then $u(t,x)\ge 0$ a.e. in $\overline{\Omega}_{0}$ for any $t\in(0, T_0]$.
  \item [(4)] If $u(0,x)>0$ a.e. in $\Omega_0$, then $u(t,x)>0$ a.e. in $\overline{\Omega}_{0}$ for any $t\in(0, T_0]$.
\end{itemize}
\end{lemma}

\begin{proof}
Let
$$
\bar h:= \sum\limits_{i,j}\max\limits_{(t,x)\in [0,T_0]\times\overline\Omega_{0}} \left| m_{ij}(t,x)\right| <+\infty
\quad \textrm{and}\quad \bar d:= \max\{d_1,\dots,d_m\}.
$$
Set  $ w(t,x):=e^{ ( \bar h+\bar d) t}u(t,x)$. Then for $t\in (0,T_{0}]$ and $x\in\overline{\Omega}_{0}$, it satisfies
\begin{equation}\label{inequality-w}
\dfrac{\partial w}{\partial t}\ge D\displaystyle\int_{\Omega_{0}}\mathcal{K}(t,x-y)w(t,y)\,\mathrm{d}y-Dw +(\bar h +\bar d)w+Mw.
\end{equation}

{\bf (1)}
Let $u(0,x)\ge0$ in $\Omega_0$. It then suffices to verify that $w(t,x)\ge 0$ for all
$(t,x)\in[0,T_{0}]\times \overline{\Omega}_{0}$. Suppose that
$$
\underline w:=\inf_{(t,x)\in [0,t_{0}]\times\overline\Omega_{0},1\le i\le m}w_{i}(t,x)<0,
$$
where $t_{0}= \frac{1}{2(2\bar h +\bar d)}$. Thus,  there exists
$\left\{(t_{k},x_{k})\right\}^{\infty}_{k=1}\in\left[0,t_{0}\right]\times\overline\Omega_{0}$
such that $w_{i_{0}}(t_{k},x_{k}) \to \underline w$  as $k \to +\infty$ for some $i_{0}$.
Integrating the $i_0$-th inequality of \eqref{inequality-w} over $[0,t_{k}]$, we have
$$
w_{i_{0}}(t_{k},x_{k})\ge w_{i_{0}}(0,x_{k})+t_{0}\left( \bar d + 2 \bar h\right) \underline w .
$$
Letting $k\to +\infty$, then $\underline w\ge t_{0}\left( \bar d+ 2 \bar h \right)\underline w$,
which is a contradiction. Therefore, $w(t,x)\ge0$ on $[0,t_{0}]\times\overline\Omega_{0}$.
Repeating this arguments, we conclude that  $w(t,x)\ge 0$ for all $(t,x)\in[0,T_{0}]\times\overline\Omega_{0}$.

{\bf (2)}
Let $u(0,\cdot)\in\hat X_{+}\setminus\hat{\mathbf 0}$. Assume that there exists $(\tilde t_{\ast},\tilde x_{\ast})\in \left( 0,T_{0}\right] \times \overline\Omega_{0}$
such that $w_{i_{0}}(\tilde t_{\ast},\tilde x_{\ast})=0$ for some $i_{0}$. Notice that
$$
\frac{\partial w_{i_{0}}}{\partial t} (\tilde t_{\ast},\tilde x_{\ast})\le  0 \quad\textrm{and}\quad\mathcal{K}_{i_{0}}(\tilde t_{\ast},0)>0.
$$
Using the cooperative property of $M$, we deduce that $w_{i_{0}}(\tilde t_{\ast},x)=0$ almost everywhere in
$\overline\Omega_{0}$. To clarify which components of $w$ are equal to zero almost everywhere at time
$\tilde t_{\ast}$, we define the following index set:
$$
\mathcal{I}:=\left\{i\in\{1,\dots,m\}: w_{i}(\tilde t_{\ast},x)=0 \textrm{ almost everywhere in } \overline\Omega_{0}\right\}.
$$
Obviously, $\mathcal{I}\ne\emptyset$, and the complementary set $\mathcal{J}:= \left\lbrace 1,\dots,m\right\rbrace \setminus\mathcal{I}$.
In view of the above discussion, we can derive that $w_{j}(\tilde t_{\ast},x)>0$ for all $x\in \overline\Omega_{0}$
if $j\in \mathcal{J}$. Since $M(t,x)$ is a continuous matrix with respect to $t$ and $x$, and $M(t,x_{0})$ is
irreducible for all $t\ge 0$, there exists some point $x^{\ast}$ such that $M(\tilde t_{\ast},x^{\ast})$ is irreducible
and $w_{i}(\tilde t_{\ast},x^{\ast})=0$ for all $i\in \mathcal{I}$. Hence, we have
$$
0\ge  \frac{\partial w_{i}}{\partial t}(\tilde t_{\ast},x^{\ast})= \sum_{j\in \mathcal{J}} m_{ij}(\tilde t_{\ast},x^{\ast})w_{j}(\tilde t_{\ast},x^{\ast}),\quad i\in \mathcal{I}.
$$
This implies
$m_{ij}(\tilde t_{\ast},x^{\ast})=0 $ for  $ i\in \mathcal{I},\, j\in \mathcal{J},
$
which contradicts the irreducibility of $M(\tilde t_{\ast},x^{\ast})$  if $\mathcal{J} $ is nonempty.
Therefore, $\mathcal{I}=\left\lbrace 1,\dots,m\right\rbrace $. Since
$$
\frac{\partial w_{i}}{\partial t}(t,x)\ge (\bar d-d_i+\bar h-m_{ii}(t,x))w_{i}\ge0, \quad (t,x)\in (0,T_{0}]\times \overline{\Omega}_{0}, \, i\in\mathcal{I},
$$
it follows that  $w_{i}(\tilde t_{\ast},x)>0$ if $w_{i}(0,x)>0$.
Consequently, $u(0,x)(=w(0,x))=0$ almost everywhere, leading to a contradiction. Hence, we have proven that for each $i=1,\dots, m$,
$w_{i}(t,x)>0$ for $(t,x)\in (0,T_{0}]\times \overline{\Omega}_{0}$.
It follows that for any fixed $t\in\left(0,T_{0}\right]$ and each $i=1,\dots, m$,
$\mathcal{W}_{i}(t,x):= d_{i}\int_{0}^{t}\int_{\Omega_{0}}\mathcal{K}_{i}(s,x-y)w_{i}(s,y)\,\mathrm{d}y\mathrm{d}s$
is a continuous function on $\overline{\Omega}_{0}$ with a positive minimum. Integrating the governing equation
for $w_i$ over $[0,t]$, we obtain $w_{i}(t,\cdot)\ge \mathcal{W}_{i}(t,\cdot)$ on $\overline{\Omega}_{0}$,
which implies $w(t,\cdot)\in \mathrm{Int}(\hat X_{+})$.

{\bf (3)}
Let $u(0,x)\ge0$ a.e. in $\Omega_0$.
Set $\tilde{u}:=u+v$ such that $\tilde{u}(0,x)\ge0$ for any $x\in\Omega_0$, where $v:=v(t,x)$ is given by Lemma \ref{le2.2-1}.
Following the approach in case (1), we deduce that $\tilde{u}(t,x)\ge 0$ in $[0,T_{0}]\times \overline{\Omega}_{0}$.
Furthermore, by Lemma \ref{le2.2-1}, it holds that $v(t,x)= 0$ a.e. in $[0,T_0]\times\overline\Omega_{0}$. Hence, it follows
that $u(t,x)\ge 0$ a.e. in $[0,T_{0}]\times \overline{\Omega}_{0}$.

{\bf (4)} The proof is similar to that of Case (3) and is therefore omitted here.
\end{proof}

Now, we establish the comparison principle for \eqref{eq2.2} and prove its well-posedness.

\begin{theorem}[Comparison principle]\label{th2.2}
Assume that {\rm\textbf{(K)}} and {\rm\textbf{(G1)}} hold.
Let $\bar u, \underline u\in C([0,T_{0}],\hat X_+)$ be a supersolution and a subsolution of the equation associated with \eqref{eq2.2}
for some $T_{0}>0$. Denote $w(t,x):=\bar u(t,x)-\underline u(t,x)$.
Then the following statements are valid:
\begin{itemize}
  \item [(1)] If $w(0,\cdot)\in\hat X_{+}$, then $w(t,\cdot)\in\hat X_{+}$ for any $t\in\left(0,T_{0}\right]$.
  \item [(2)] If {\rm\textbf{(G2)}} holds and $w(0,\cdot)\in\hat X_{+}\setminus\hat{\mathbf 0}$, then $w(t,\cdot)\in \mathrm{Int} (\hat X_{+})$ for any $t\in\left(0,T_{0}\right]$.
  \item [(3)] If $w(0,x)\ge0$ a.e. in $\Omega_0$, then $w(t,x)\ge 0$ a.e. in $\overline{\Omega}_{0}$ for any $t\in(0, T_0]$.
  \item [(4)] If {\rm\textbf{(G2)}} holds and $w(0,x)>0$ a.e. in $\Omega_0$, then $w(t,x)>0$ a.e. in $\overline{\Omega}_{0}$ for any $t\in(0, T_0]$.
\end{itemize}
\end{theorem}

\begin{proof}
By assumption \textbf{(G1)}, there exists $H(t,x,\bar u, \underline u) $ which is cooperative and satisfies
$g(t,x,\bar u)- g(t,x,\underline u)= H(t,x,\bar u, \underline u)w$.
Applying Lemma \ref{le2.2} to the system for $w$, we deduce the desired conclusions.
\end{proof}

It is worth noting that the above maximum and comparison principles remain valid on $[s,s+T_0]$ if we consider
$s$ as the initial time. This fact will be utilized when needed in the following proofs.

\begin{proposition}[Well-posedness]\label{2.1}
Assume that {\rm\textbf{(K)}} and {\rm\textbf{(G1)}} hold. For each $u_{0}\in \hat X_{+}^{v}$, the problem \eqref{eq2.2}
admits a unique solution $u\in C( \mathbb{R}_{+},\hat X_{+}^{v})$ such that
$\frac{\partial u}{\partial t}(t,x)$ exists for $t\ge0$ and $x\in \overline \Omega_{0}$.
\end{proposition}

\begin{proof}
For any given $t_{0}>0$, let $Y_{t_{0}}:=C([0,t_{0}], \hat X)$ equipped with the norm
$\left\|u\right\|_{Y_{t_{0}}}:=\sup\limits_{t\in[0,t_0]}\left\|u(t,\cdot)\right\|_{\hat X}$.
Define an operator $\Gamma_{t_{0}}$ on $Y_{t_{0}}$ as
\[
 \begin{aligned}
\left[ \Gamma_{t_{0}}\phi\right](t,x) :=u_{0}(x)+\int_{0}^{t}\left[D\int_{\Omega_{0}}\mathcal{K}(s,x-y)\phi(s,y)\,\mathrm{d}y-D\phi(s,x)
+g(s,x,\phi(s,x))\right]\mathrm{d}s.
\end{aligned}
\]
Set
$a:=2\left\| v\right\|_{\mathbb{R}^m} $,
and define
$$
Y^{a}_{t_{0}}:=\left\lbrace \phi \in Y_{t_{0}}: \left\| \phi \right\| _{Y_{t_{0}}} \le a\right\rbrace.
$$
Obviously, there exists $\delta>0$ such that $\Gamma_{t_{0}}Y^{a}_{t_{0}} \subset Y^{a}_{t_{0}}$ for any $t_{0}\in [0,\delta ]$.
Let
$\bar d:= \max\left\{d_{1},\dots, d_{m}\right\}$. By assumption \textbf{(G1)}, it is straightforward to verify that
\[
\left\|\Gamma_{t_{0}}\phi_{1}-\Gamma_{t_{0}}\phi_{2}\right\|_{Y_{t_{0}}}\le(2\bar d+M)t_{0}\left\|\phi_{1}-\phi_{2}\right\|_{Y_{t_{0}}},
\]
where $M=M(\delta,a)>0$ is Lipschitz constant of $g$. Let $\delta_{0}:= \frac{1}{2(2\bar d+M)}$.
Then for any $t_{0}\in [0,\delta_{0}]$, $ \Gamma_{t_{0}}$ is a contraction mapping on $ Y^{a}_{t_{0}}$,
which implies the existence of a unique $u\in Y^{a}_{\delta_{0}}$ satisfying $\Gamma_{\delta_{0}}u =u $.
It follows that system \eqref{eq2.2} admits a unique local mild solution $u\in Y_{\delta_{0}} $ in the sense that
\[
u(t,x)= u_{0}(x)+\int_{0}^{t}\left[D\int_{\Omega_{0}}\mathcal{K}(s,x-y)u(s,y)\,\mathrm{d}y-Du(s,x) +g(s,x,u(s,x))\right]\mathrm{d}s.
\]
By iterating the above arguments, we can extend the solution to the maximal  interval of existence $\left[ 0,T_{\rm max}\right) $
satisfying
$$
\limsup\limits_{t\to T_{\rm max}^{-}}\left\| u(t,\cdot)\right\|_{\hat X}=+\infty \quad \textrm{if}\quad T_{\rm max}< +\infty.
$$
Applying Lemma \ref{le2.2}, we conclude that $u(t,x)\le v$ for all $(t,x)\in \left[ 0,T_{\rm max}\right) \times \overline\Omega_{0}$,
which yields $T_{\rm max}=+\infty$.
It is easy to verified that $u(t,x)$ is a uniformly continuous function with respect to $t$ for each fixed $x$,
and $\frac{\partial u}{\partial t}(t,x)$ exists for $t\ge0$ and $x\in \overline \Omega_{0}$.
\end{proof}

\begin{remark}
If $u_{0}\in X_{+}^{v} $, a similar argument shows that \eqref{eq2.2} admits a unique solution
$u\in  C^{1,0}(\mathbb{R}_{+}\times \overline\Omega_{0}, \mathbb{R}^{m} )$.
\end{remark}

\begin{remark}\label{pr3.2}
Since system \eqref{eq2.2} admits the comparison principle,
the arguments similar to those for \cite[Lemma 2.1]{MR4356639}
imply that any locally attractive solution of system \eqref{eq2.2} is Liapunov stable.
\end{remark}

\subsection{Threshold dynamics of \eqref{eq2.2} in the time-periodic case}

In this subsection, we propose a general
framework for rigorously analyzing the threshold dynamics of \eqref{eq2.2} in the time-periodic case.
We need the following additional assumptions on $\mathcal{K}$ and $g$:
\begin{itemize}
\item[\textbf{(T)}]  For each $1\le i\le m$, $\mathcal{K}_{i}(t,x)=\mathcal{K}_{i}(t+T,x)$
 and $g_{i}(t,x,u)=g_{i}(t+T,x,u)$;
\item[\textbf{\textbf{(G3)}}] $\alpha g(t,x,u)< g(t,x, \alpha u)$ for $u\in \mathbb{M}$
with $u\in\mathrm{Int}(\mathbb{R}^{m}_{+})$, $t\ge 0$, $x\in \overline \Omega_{0}$ and $\alpha\in (0,1)$;
\item[\textbf{\textbf{(G3$'$)}}] $\alpha g(t,x,u)\ll g(t,x, \alpha u)$ for $u\in \mathbb{M}$
with $u\in\mathrm{Int}(\mathbb{R}^{m}_{+})$, $t\ge 0$, $x\in \overline \Omega_{0}$ and $\alpha\in (0,1)$.
\end{itemize}
When $m=1$, the assumptions \textbf{(G3)} and \textbf{(G3$'$)} are equivalent.
Define
$$
\mathcal{C}_{T}:=\left\lbrace u\in C(\mathbb{R}\times \overline\Omega_{0},\mathbb{R}^{m}): \right. \left.  u(t,x)=u(t+T,x)\right\rbrace
$$
equipped with the norm
$\left\| u\right\|_{\mathcal{C}_{T}}:=\sup\limits_{t\in[0,T]} \left\| u(t,\cdot)\right\|_{X}$,
and the cone
$$
\mathcal{C}_{T}^{+}:=\left\lbrace  u\in \mathcal{C}_{T}: u(t,x)\ge 0 \right\rbrace.
$$
Let $\tilde X$ denote the set of all bounded and measurable functions $u: \mathbb{R}\times \overline\Omega_{0}\to\mathbb{R}^{m} $
that satisfy $u(t,x)=u(t+T,x)$ for all $(t,x)$. This space is equipped with norm
$\left\| u\right\| _{\tilde X}:=\sup_{t\in [0,T]} \left\| u(t,\cdot)\right\|_{\hat X}$.
The corresponding positive cone is defined as
$$
\tilde X_{+}:= \left\lbrace  u\in \tilde X:  u(t,x)\ge 0,\,(t,x)\in\mathbb{R}\times \overline\Omega_{0}\right\rbrace.
$$
Similarly, the partial order induced by the cone can be defined for $\mathcal{C}_{T} $ and $\tilde X$.

For any given $a\in \mathbb{R}^{m}$, let
$$
\tilde X_{+}^{a}:=\left\lbrace u\in \tilde X: 0\le  u(t,x)\le a,\,(t,x)\in\mathbb{R}\times \overline\Omega_{0}\right\rbrace.
$$
Define
\[
\left[ L_{T}u\right](t,x) := -\frac{\partial}{\partial t} u(t,x) +D\int_{\Omega_{0}}\mathcal{K}(t,x-y)u(t,y)\,\mathrm{d}y-Du+
\left( \dfrac{\partial g_{i}(t,x,0)}{\partial u_{j}}\right)_{m\times m}  u(t,x),
\]
for $u\in\mathcal{C}_{T}$.
We denote the spectrum of $L_{T}$ by $\sigma\left( L_{T} \right)$. The spectral bound of $L_{T}$ is given by
$s(L_{T}):={\sup} \left\lbrace \text{Re} \mu : \mu \in  \sigma \left( L_{T}\right)\right\rbrace$.
Let $\Phi(t,s)$ denote the evolution family on $X$ generated by the system
\[
{\begin{array}{*{20}{l}}
\dfrac{\partial u}{\partial t}=D\displaystyle\int_{\Omega_{0}}\mathcal{K}(t,x-y)u(t,y)\,\mathrm{d}y-Du
+\left( \dfrac{\partial g_{i}(t,x,0)}{\partial u_{j}}\right)_{m\times m} u, & t>0, x\in \Omega_{0}.
\end{array}}
\]
The exponential growth bound of evolution family $\Phi(t,s)$ is defined as
\[
\omega \left( \Phi \right) := \inf \left\{ {\tilde \omega :\exists M \ge 1 \textrm{ such that } \left\|
{\Phi\left( {t,s} \right)} \right\| \le M{e^{\tilde \omega \left( {t - s} \right)}}},\,  \forall t,s \in \mathbb{R},t \ge s\right\}.
\]
By applying \cite[Proposition 5.5 and Lemma 5.8]{MR2505085}, along with the discussion in \cite[Section 2.2]{MR4620151}, we have
\begin{equation}\label{eq2.0}
	s(L_{T}) =\omega(\Phi)=\frac{{\ln r\left( {\Phi\left( {T,0} \right)} \right)}}{T},
\end{equation}
where $r\left( \Phi\left( {T,0} \right)\right) $ denotes the spectral radius of $\Phi\left( {T,0} \right) $.

To investigate the threshold dynamics of system \eqref{eq2.2} in the time-periodic case, we propose the following
periodic problem associated with system \eqref{eq2.2}:
\begin{equation}\label{eq2.1}
\left\{
    {\begin{array}{*{20}{l}}
			\dfrac{\partial u}{\partial t}=D\displaystyle\int_{\Omega_{0}}\mathcal{K}(t,x-y)u(t,y)\,\mathrm{d}y-Du +g(t,x,u),
            & t\in \mathbb{R}, x\in \Omega_{0}, \\[4mm]
			u(t,x)=u(t+T,x),  & t\in \mathbb{R}, x \in \Omega_{0}.\\
	\end{array}} \right.
\end{equation}
Since the solution mapping of \eqref{eq2.2} lacks compactness in $C(\overline\Omega_{0},\mathbb{R}^{m})$,
only the pointwise convergence of solutions can be examined as $t\to+\infty$. This obstacle motivates us to
analyze \eqref{eq2.1} in the space $\tilde X_{+}^{v}$. Furthermore, in the subsequent proof, we introduce a
method for constructing subsolutions in the absence of a principal eigenvalue, which can also be extended to
address other similar problems where the principal eigenvalue does not exist.

\begin{lemma}\label{le2.1}
Assume that {\rm\textbf{(K)}}, {\rm\textbf{(G1)--(G3)}} and {\rm\textbf{(T)}} hold. The following statements are valid for system \eqref{eq2.1}:
\begin{enumerate}
	\item[(1)]  If $\omega \left( \Phi \right)<0$, then $0$ is the unique nonnegative solution in $\tilde X_{+}^{v}$;
	\item[(2)]  If $\omega \left(\Phi\right)=0$, and either {\rm\textbf{(G3$'$)}} holds, or $\omega \left(\Phi\right)=0$ is the principal eigenvalue,
                then $0$ is the unique nonnegative solution in $\tilde X_{+}^{v}$;
	\item[(3)]  If $\omega \left( \Phi \right)>0$, then \eqref{eq2.1} admits a unique positive solution $w^*\in\tilde X_{+}^{v}$.
Moreover, $w^*\in\mathrm{Int}(\mathcal{C}_{T}^{+})$.
\end{enumerate}
\end{lemma}

\begin{proof}
It is evident that \eqref{eq2.1} admits at least one trivial solution. Let $u$ be a $T$-periodic solution of \eqref{eq2.1}
in $\tilde X_{+}^{v}$. The following proof is divided into three cases:
$\omega \left( \Phi \right)< 0$, $\omega \left( \Phi \right)= 0$, and $\omega \left( \Phi \right)> 0$.

{\bf Case 1}. When $\omega \left( \Phi \right)< 0$, we choose $\bar u\in X$ such that $u(0,x)\le \bar u(x)$.
By Theorem \ref{th2.2}, it holds that $u(t,x)\le u(t,x;\bar u)$.
Since $\alpha g(t,x,u)\le g(t,x,\alpha u)$ for $\alpha\in [0,1]$, applying Theorem \ref{th2.2} again yields
$u(t,x;\bar u)\le  \left[ \Phi(t,0)\bar u \right](x)$.
Thus, we have
$$
\left\| u(t, \cdot)\right\|_{\hat X}\le \left\| \Phi(t,0)\bar u \right\|_{X} \to 0 \quad \textrm{as}\quad t\to+\infty,
$$
which implies $u\equiv0$.

{\bf Case 2}. When $ \omega \left( \Phi \right)= 0$, define:
$$
\underline u_{i}:= \inf_{(t,x)\in [0,T]\times\overline\Omega_{0}}u_{i}(t,x)\ge 0,\quad i=1,\dots,m,
$$
and the index sets:
$$
\mathcal{I}:=\big\{i\in\{1,\dots,m\}: \underline u_{i}=0\big\}, \qquad
\mathcal{J}:=\big\{j\in\{1,\dots,m\}: \underline u_{j}>0\big\}.
$$
We claim that $\mathcal{J}=\emptyset$. Suppose, for contradiction, that this is not the case.
By Theorem \ref{th2.2} and the temporal periodicity of $u$,
there exists $c_1>0$ such that $\min\{\underline u_{1}, \dots,\underline u_{m}\}\ge c_{1}$.
Below, we proceed with the proof to derive a contradiction, assuming that either \textbf{(G3$'$)} holds
or $\omega \left(\Phi\right)=0$ is the principal eigenvalue.

Assume that \textbf{(G3$'$)} holds. Then there exists $c_2=c_2(c_1)>0$, independent of $t,x,i$ and $u$,
such that
$$
\sum_{j=1}^{n}\frac{\partial g_{i}(t,x,0)}{\partial u_{j}}u_{j}-g_{i}(t,x,u)\ge c_{2}, \quad i=1,\dots,m.
$$
We obtain
$$
\dfrac{\partial u}{\partial t}\le D\displaystyle\int_{\Omega_{0}}\mathcal{K}(t,x-y)u(t,y)\,\mathrm{d}y-Du
+\left( \dfrac{\partial g_{i}(t,x,0)}{\partial u_{j}}\right)_{m\times m}u-c_3u,
$$
where $c_{3}=\frac{c_{2}}{\left\| u\right\| _{\tilde X}}$. It follows that $r(e^{-c_{3}T}\Phi(T,0))< r(\Phi(T,0))=1$.
By the same arguments as those in the case where  $\omega(\Phi)<0$, we conclude $u\equiv0$,
which is a contradiction.

Next, we assume that $\omega(\Phi)=0$ is the principal eigenvalue.
Let $p$ be the principal eigenfunction associated with $\omega(\Phi)$ such that
$0\ll p\ll u$ in $\mathbb{R}\times\overline\Omega_0$.
By \textbf{(G3)}, we have
\begin{align*}
    0&=\omega(\Phi)p(t,x)
    \\
	&  =  \dfrac{\partial p}{\partial t}-D\displaystyle\int_{\Omega_{0}}\mathcal{K}(t,x-y)p(t,y)\,\mathrm{d}y+Dp
-\left( \dfrac{\partial g_{i}(t,x,0)}{\partial u_{j}}\right)_{m\times m}p
	\\
	&  \le  \dfrac{\partial p}{\partial t}-D\displaystyle\int_{\Omega_{0}}\mathcal{K}(t,x-y)p(t,y)\,\mathrm{d}y+Dp
-g(t,x,p)
\end{align*}
in $\mathbb{R}\times\overline\Omega_0$.
Define $\alpha^*:=\sup\{\alpha>0: \alpha u\le p\ \text{in}\ \mathbb{R}\times\overline\Omega_0\}$ and
set $\xi:=p-\alpha^*u$. Then $\alpha^*\in(0,1)$ and $\xi\in\tilde X_{+}^{v}$ satisfying
\begin{align}\label{contradiction}
\dfrac{\partial\xi}{\partial t}(t,x)&\ge D\displaystyle\int_{\Omega_{0}}\mathcal{K}(t,x-y)\xi(t,y)\,\mathrm{d}y-D\xi+g(t,x,p)-\alpha^*g(t,x,u) \nonumber
    \\
	&  > D\displaystyle\int_{\Omega_{0}}\mathcal{K}(t,x-y)\xi(t,y)\,\mathrm{d}y-D\xi+g(t,x,p)-g(t,x,\alpha^*u) \nonumber
    \\
	&  = D\displaystyle\int_{\Omega_{0}}\mathcal{K}(t,x-y)\xi(t,y)\,\mathrm{d}y-D\xi+H(t,x,p,\alpha^*u)\xi,
\end{align}
where $H(t,x,p,\alpha^*u)$ is $T$-periodic, cooperative and irreducible. Suppose there exists $t_0$ such that $\xi(t_0,\cdot)\in\hat X_{+}\setminus\hat{\mathbf 0}$.
By Lemma \ref{le2.2} and the temporal periodicity of $\xi$, it follows that $\xi(t,\cdot)\in \mathrm{Int} (\hat X_{+})$ for all $t\in\mathbb{R}$ which contradicts the definition of $\alpha^*$.
Otherwise, we must have $\xi(t,\cdot)\in\hat X_{+}\cap\hat{\mathbf 0}$ for all $t\in\mathbb{R}$.
In this case, there exists $(t_0,x_0)\in\mathbb{R}\times\overline\Omega_0$ such that $\xi(t_0,x_0)=0$. Considering
\eqref{contradiction} at $(t_0,x_0)$, we again arrive at a contradiction.

According to Theorem \ref{th2.2}, we conclude that for any $t\in\mathbb{R}$, $u(t,x)=0$ almost everywhere for
$x\in \overline\Omega_{0}$. Assume that there exists $(t^{\ast},x^{\ast})\in\mathbb{R}\times\overline\Omega_{0}$
such that $u(t^{\ast},x^{\ast})>0$. It follows that $u(t,x^{\ast})>0$ for all $t\in\mathbb{R}$ and satisfies
the following periodic differential equation:
$$
\left\{
    {\begin{array}{*{20}{l}}
		\dfrac{\mathrm{d}}{\mathrm{d}t}u(t,x^{\ast})=-Du +g(t,x^{\ast},u), & t>0,  \\[3mm]
		u(t,x^{\ast})=u(t+T,x^{\ast}),  & t\in \mathbb{R}.\\
    \end{array}}
\right.
$$
For each $x\in \overline\Omega_{0}$, define the operator:
$$
Q_{x}(t)v:=e^{-Dt+\int_{0}^{t}(\frac{\partial g_{i}(s,x,0)}{\partial u_{j}})_{m\times m}\,\mathrm{d}s } v,
\quad v \in \mathbb{R}^{m}.
$$
According to \cite[Lemma B.2]{MR3705788} and \eqref{eq2.0}, we have
$\max\limits_{x\in \overline\Omega_{0}}r(Q_{x}(T)) \le r(\Phi(T,0))=1$.
By \cite[Theorem 2.3.4]{MR3643081}, we can show that $\left\| u(t,x^{\ast})\right\|_{\mathbb{R}^m}\to0$
as $t\to+\infty$
which leads to a contradiction. Thus, we conclude that $u(t,x)\equiv 0$ for all $(t,x)\in\mathbb{R}\times \overline\Omega_{0}$.

{\bf Case 3}. When $\omega \left( \Phi \right)> 0$, we first show that the positive solution is unique, if it exists.
Suppose that $w$ is another positive solution of \eqref{eq2.1} in $\tilde X_{+}^{v}$. Based on the previous analysis,
it holds that
$$
\inf_{(t,x)\in [0,T]\times\overline\Omega_{0},i=1,\dots,m} w_{i}(t,x)>0.
$$
Define $\beta^*:=\sup\{\beta>0: \beta u\le w\ \text{in}\ \mathbb{R}\times\overline\Omega_0\}$ and
set $\tilde{\xi}:=w-\beta^*u$. Without loss of generality, assume $\beta^*\in(0,1]$.
Then, applying a similar argument on the equation of $\tilde{\xi}\in\tilde X_{+}^{v}$ as in Case 2, we obtain $\beta^*=1$,
that is, $u\le w$ in $\mathbb{R}\times\overline\Omega_0$. Similarly, we can also show $u\ge w$ in $\mathbb{R}\times\overline\Omega_0$.
Thus, the positive solution of \eqref{eq2.1} is unique.

Next, we prove the existence of a positive solution for \eqref{eq2.1}.
If there exists a function $\phi\in X_{+}^{v}$ such that $\Phi(T,0)\phi\ge \phi$, then, following the similar arguments
as in \cite[Theorem 2.2]{MR3968264}, it can be shown that \eqref{eq2.1} admits a positive solution $w \in\tilde X_{+}^{v}$.
Furthermore, in view of the discussion on \cite[Theorem 2.2]{MR3968264} (or \cite[Theorem E]{MR3000610}) and the
uniqueness of positive solution, we also have $w\in\mathrm{Int}(\mathcal{C}_{T}^{+})$.

We proceed to construct a function $\phi\in X_{+}^{v}$ such that $\Phi(T,0)\phi\ge \phi$.
If $r(\Phi(T,0))>\max\limits_{x\in \overline\Omega_{0}}r(Q_{x}(T)) $, based on the results of
\cite[Lemmas 2.5 and B.2]{MR3705788} (or \cite[Theorem 2.1]{MR3637938}),
we conclude that there exists $\phi\in\mathrm{Int}(\mathcal{C}_{T}^{+})$ with $\left\|\phi\right\|=1$ satisfying $
L_{T}\phi= s(L_{T})\phi$.
Since
$$
\lim_{\delta \to 0^{+}}\left\|  \frac{g(t,x, \delta\phi(t,x))}{\delta}- \left(\frac{\partial g_{i}(t,x,0)}{\partial u_{j}}\right)_{m\times m}\phi(t,x)\right\| _{\mathbb{R}^{m}}  = 0
$$
uniformly for $(t,x)\in [0,T]\times \overline\Omega_{0}$, it follows that there exists $\delta>0$ such that
$$
g(t,x,\delta \phi)\ge  \left(\frac{\partial g_{i}(t,x,0)}{\partial u_{j}}\right)_{m\times m}\delta\phi-s(L_{T})\delta\phi
$$
for all $(t,x)\in [0,T]\times \overline\Omega_{0}$. As a result, the following inequality holds:
\[
\delta \frac{\partial }{\partial t}\phi(t,x)\le D\delta\int_{\Omega_{0}}\mathcal{K}(t,x-y)\phi(t,y)\,\mathrm{d}y
- D\delta \phi + g(t,x,\delta \phi),\quad t\in[0,T].
\]
Thus, we conclude that $\Phi(T,0)\delta\phi(0,\cdot)\ge \delta\phi(0,\cdot)$.

If $r(\Phi(T,0))=\max\limits_{x\in \overline\Omega_{0}}r(Q_{x}(T))$, there exist $x_{0}\in \overline \Omega_{0}$
and $\varrho>0$ such that $r(Q_{x}(T))>1$  for all $x\in B_{x_{0},\varrho}$.
As stated in \cite[Section II]{MR1335452}, the function $r(x):=r(Q_{x}(T))$ is continuous,
and there exists a continuous function $\varphi(x)$ satisfying
$Q_{x}(T)\varphi(x)=r(x)\varphi(x)$ for $x\in \overline\Omega_{0}$.
Define $r_{0}:= \min\limits_{x\in B_{x_{0},\frac{2}{3}\varrho}}r(x)>1$,
and construct the cut-off function
\[
\eta(x):=  \begin{cases}
	1, &  \left\| x-x_{0}\right\|_{\mathbb{R}^n}\le \dfrac{1}{3}\varrho , \\[2mm]
	1- \dfrac{3}{\varrho}(\left\|x-x_{0} \right\|_{\mathbb{R}^n}- \dfrac{1}{3}\varrho ), & \dfrac{1}{3}\varrho\le \left\| x-x_{0}\right\|_{\mathbb{R}^n}\le \dfrac{2}{3}\varrho,\\[2mm]
	0, & \left\| x-x_{0}\right\|_{\mathbb{R}^n}\ge \dfrac{2}{3}\varrho.
\end{cases}
\]
By assumption \textbf{(G2)}, we can choose $\varphi(x)$ such that
$$
\max_{x\in B_{x_{0},\frac{2}{3}\varrho}}\left\| \varphi(x)\right\|_{\mathbb{R}^m} = 1
\quad\textrm{and}\quad
\min_{x\in B_{x_{0},\frac{2}{3}\varrho}}\left| \varphi_{i}(x)\right|>0 \quad\textrm{for all}\quad i.
$$
Define $\psi(t,x):= e^{-\frac{\ln r_{0}}{T}t}Q_{x}(t)\varphi(x)$.
Similarly, there exists $\delta'>0$ such that
$$
g(t,x,\delta' \psi)\ge  \left(\frac{\partial g_{i}(t,x,0)}{\partial u_{j}}\right)_{m\times m}\delta'\psi-\frac{\ln r_{0}}{T}\delta'\psi
$$
for all $(t,x)\in [0,T]\times \overline\Omega_{0}$. Since $\eta(x)\in [0,1]$, it follows that
$g(t,x,\delta' \eta\psi) \ge \eta g(t,x,\delta' \psi)$.
Thus, for $t\in [0,T]$ and $x\in \overline\Omega_{0}$, we have
\[
\delta'\eta(x) \frac{\partial }{\partial t}\psi(t,x)\le  D\delta'\int_{\Omega_{0}}\mathcal{K}(t,x-y)\eta(y)\psi(t,y)\,\mathrm{d}y- D\delta'\eta \psi + g(t,x,\delta' \eta \psi).
\]
Consequently, we obtain $\Phi(T,0)\delta'\eta\psi(0,\cdot)\ge \delta'\eta\psi(T,\cdot) \ge \delta'\eta\psi(0,\cdot)$.
This completes the proof.
\end{proof}

We are now ready to prove the main results of this section.
\begin{theorem}\label{th2.3}
Assume that {\rm\textbf{(K)}}, {\rm\textbf{(G1)--(G3)}} and {\rm\textbf{(T)}} hold. The following statements are valid for system \eqref{eq2.2}:
\begin{enumerate}
\item[(1)]  If $\omega(\Phi)<0$, then $0$  is globally asymptotically stable in $X_{+}^{v}$.
\item[(2)]  If $\omega \left(\Phi\right)=0$ and either {\rm\textbf{(G3$'$)}} holds or $\omega \left(\Phi\right)=0$ is the principal eigenvalue, then $0$ is globally asymptotically stable in $X_{+}^{v}$.
\item[(3)] If $\omega(\Phi)>0$, then $w^{\ast}$ is globally asymptotically stable in
    $X_{+}^{v}\setminus\left\lbrace 0\right\rbrace$, where $w^{\ast}$ is obtained in Lemma \ref{le2.1}(2).
\end{enumerate}
\end{theorem}

\begin{proof}
In the case where $ \omega(\Phi)\le 0 $, for any $ \phi\in X_{+}^{v}$, by comparison principle, we have
$0\le u(t,x;\phi)\le u(t,x;v)$.
It is clear that
$$
u(kT+t,x;v)\le u((k-1)T+t,x;v)\le v, \quad k\ge1.
$$
Then the sequence $\{u(kT+t,x;v)\}$ converges pointwise as $k\to+\infty$
to some function $u_{\ast}\in \tilde X_{+}^{v} $. By Lebesgue dominated convergence theorem, it is straightforward to verify that
$u_*$ is a nonnegative solution of \eqref{eq2.1}. Furthermore, by Lemma \ref{le2.1}, we deduce that $u_*(t,x)\equiv0$. Subsequently,
applying Dini's theorem, we conclude that $u(t,x;v)\to 0$
as $t \to +\infty$ uniformly for $x\in \overline\Omega_{0}$. Therefore, $u(t,x;\phi)$ also converges to $0$
uniformly for $x\in \overline\Omega_{0}$ as $t\to+\infty$. The desired result
immediately follows from \cite[Lemma 2.1]{MR4356639}.

In the case where $\omega(\Phi)>0$, by Lemma \ref{le2.1},  \eqref{eq2.1} admits a unique positive solution $w^{\ast}(t,x)$.
For any $\phi\in X_{+}^{v}$ with $\phi\neq 0$, Theorem \ref{th2.2} implies that
$u(T,x;\phi)\in\mathrm{Int}(X_{+})$.
Thus, there exists $0<\delta<1$ such that
$u(T,x;\phi)\ge \delta w^{\ast}(0,x)$.
Since $\delta w^{\ast}(t,x)$ is a subsolution of the equation corresponding to \eqref{eq2.2} due to \textbf{(G3)}, we have
$\delta w^{\ast}(t,x) \le u(t,x;\delta w^*(0,\cdot))$.
It follows that, for $k\ge 1$,
$$
u((k+1)T+t,x;\phi)\ge u(kT+t,x;\delta w^{\ast}(0,\cdot)) \ge u((k-1)T+t,x;\delta w^{\ast}(0,\cdot)).
$$
According to Lemma \ref{le2.1}, it is easy to see that the sequences $\{u(kT+t,x;v)\}$
and $\{u(kT+t,x;\delta w^*(0,\cdot))\}$ converge pointwise to the same function $w^*\in\mathrm{Int}(\mathcal{C}_{T}^{+})$ as $k\to+\infty$ .
Then, by Dini's theorem, we have
$$
\lim_{t\to + \infty}\left\| u(t,\cdot;v) -w^{\ast}(t,\cdot)\right\|_{X} =\lim_{t\to +\infty}\left\|u(t,\cdot;\delta w^{\ast}(0,\cdot)) -w^{\ast}(t,\cdot)\right\|_{X} =0,
$$
which yields
$$
\lim_{t\to +\infty}\left\| u(t,\cdot;\phi) -w^{\ast}(t,\cdot)\right\|_{X} =0.
$$
With Remark \ref{pr3.2}, we then obtain the global asymptotic stability of $w^{\ast}$.
\end{proof}

Next, we present a generalized version of Theorem \ref{th2.3}, which provides the global dynamics of \eqref{eq2.2} in $\hat X$.

\begin{theorem}\label{generalized-dynamical}
Assume that {\rm\textbf{(K)}}, {\rm\textbf{(G1)--(G3)}} and {\rm\textbf{(T)}} hold. The following statements are valid for system \eqref{eq2.2}:
\begin{enumerate}
\item[(1)]  If $\omega(\Phi)<0$, then $0$ is globally asymptotically stable in $\hat X_{+}^{v} $;
\item[(2)]  If $\omega \left(\Phi\right)=0$ and either {\rm\textbf{(G3$'$)}} holds or $\omega \left(\Phi\right)=0$ is the principal eigenvalue, then $0$ is globally asymptotically stable in $\hat X_{+}^{v}$.
\item[(3)] If $\omega(\Phi)>0$, then $w^{\ast}$  is globally asymptotically stable in $\hat X_{+}^{v}\setminus\hat{\mathbf 0}$,
 where $w^{\ast}\in\mathrm{Int}(\mathcal{C}_{T}^{+})$ as obtained in Lemma \ref{le2.1}(2).
\end{enumerate}
\end{theorem}

\begin{proof}
In the case where $\omega(\Phi)\le0$, the proof is similar to that of Theorem \ref{th2.3}.
For the case where $\omega(\Phi)>0$, according to Theorem \ref{th2.2}, the solution of system \eqref{eq2.2} in
$\hat X_{+}^{v}\setminus\hat{\mathbf 0}$ becomes strongly positive at any positive time.
We may consider a positive time as the initial time. The subsequent proof is analogous to that of
Theorem \ref{th2.3} and is therefore omitted.
\end{proof}

Finally, we consider a special case where $\mathcal{K} $ and $g$ are independent of $t$, specifically:
\begin{itemize}
	\item[\textbf{(T0)}]  $\mathcal{K}(t,x)=\mathcal{K}(x)$ and $g(t,x,u)=g(x,u)$.
\end{itemize}
In this case,  $L_{T}$ and problem \eqref{eq2.1} reduce to
\[
\left[  L_{0}u\right](x) := D\int_{\Omega_{0}}\mathcal{K}(x-y)u(y)\,\mathrm{d}y-Du+ \left( \dfrac{\partial g_{i}(x,0)}{\partial u_{j}}\right)_{m\times m}  u(x), \quad  u\in X,
\]
and
\begin{equation}\label{eq2.3}
D\int_{\Omega_{0}}\mathcal{K}(x-y)u(t,y)\,\mathrm{d}y-Du +g(x,u)=0,
\end{equation}
respectively. According to \cite[Theorem 3.14]{MR2505085}, \cite[Corollary 2.1]{MR4628895} or
\cite[Proposition 2.4]{MR1641201}, it holds that $s(L_{0})=s(-\frac{\partial }{\partial t} +L_{0})$.
Here, $-\frac{\partial }{\partial t} +L_{0} $ is considered as an operator on $\mathcal{C}_{T}$.
Based on Theorem \ref{th2.3}, we can deduce the following conclusions.

\begin{corollary}\label{c2.1}
Assume that {\rm\textbf{(K)}}, {\rm\textbf{(G1)--(G3)}} and {\rm\textbf{(T0)}} hold. The following statements are valid for system \eqref{eq2.2}:
\begin{enumerate}
\item[(1)]  If $s(L_{0})<0$, then $0$ is globally asymptotically stable for system \eqref{eq2.2} in $X_{+}^{v} $.
\item[(2)]  If $s(L_{0})=0$ and either {\rm\textbf{(G3$'$)}} holds or $s(L_{0})=0$ is the principal eigenvalue,  then $0$ is globally asymptotically stable for system \eqref{eq2.2} in $X_{+}^{v} $.
\item[(3)] If $s(L_{0})> 0  $, then problem \eqref{eq2.3} admits a unique positive solution $w^{\ast}$ in
    $\hat X_{+}^{v}$. Moreover, $w^{\ast}\in\mathrm{Int}(X_{+})$ and is globally asymptotically stable for system \eqref{eq2.2}
    in $X_{+}^{v}\setminus\left\lbrace 0\right\rbrace$.
\end{enumerate}
\end{corollary}

\section{Asymptotically bounded domain}

In this section, we investigate the global dynamics of system \eqref{intro-4} when the domain \(\Omega_t\) is
asymptotically bounded, with a particular focus on two cases where \(\Omega_t\) converges asymptotically to a
fixed domain or a time-periodic domain.

\subsection{Asymptotically fixed domain}
In this subsection, we consider the case where the domain $\Omega_t$ asymptotically converges to a fixed domain.
Specifically, we assume that $\rho (t)$ satisfies the following asymptotic condition:
\begin{itemize}
\item [\textbf{(B1)}] $\lim\limits_{t\to+ \infty} \rho (t)= \rho_\infty>0,\,\,\, \lim\limits_{t\to+ \infty} \dot\rho(t)=0$.
\end{itemize}

 We start with the limiting system:
\begin{equation}\label{eq3.1}
\left\{
    {\begin{array}{*{20}{l}}
	   \dfrac{\partial u}{\partial t}=D\displaystyle\int_{\Omega_{0}}J_{\frac{1}{\rho_{\infty}}}(y-z)u(t,z)\,\mathrm{d}z-Du(t,y) +f(u), & t>0, y \in \Omega_{0},
        \\[4mm]
	   u(0,y)=u_{0}(y),  & y \in \Omega_{0}.\\
    \end{array}}
\right.
\end{equation}
Define
\[
\left[ Lu\right](y):= D\int_{\Omega_{0}}J_{\frac{1}{\rho_{\infty}}}(y-z)u(z)\,\mathrm{d}z-Du(y) +{\mathcal{D}f(0)}u(y), \quad  u \in X.
 \]
Let
\[
\left[\bar L_{\varepsilon}u\right](y):= D\int_{\Omega_{0}}\left[ J_{\frac{1}{\rho_{\infty}}}(y-z)+\varepsilon \right] u(z)\,\mathrm{d}z-Du(y)
    +\varepsilon u(y)+\mathcal{D}f(0)u(y), \quad  u \in X,
\]
and
\[
\left[ \underline L_{\varepsilon}u\right](y):= D\int_{\Omega_{0}}\left[ \max\left\{J_{\frac{1}{\rho_{\infty}}}(y-z)-\varepsilon,0\right\}\right] u(z)\,\mathrm{d}z-Du(y)- \varepsilon u(y) +\mathcal{D}f(0)u(y), \  u \in X.
\]
According to \cite{MR4628895,MR3637938}, $L$ admits principal eigenvalue, as do $\bar L_{\varepsilon}$ and $\underline L_{ \varepsilon}$.
Let $\lambda^{\ast}:=s(L), \bar \lambda_{\varepsilon}^{\ast}:=s(\bar L_{ \varepsilon})$ and
$\underline\lambda_{\varepsilon}^{\ast}:=s(\underline L_{ \varepsilon})$ which are the principal eigenvalues
of $L, \bar L_{ \varepsilon}$ and $\underline L_{ \varepsilon} $, respectively. It is easy to verify that
$\bar \lambda_{\varepsilon}^{\ast}\to \lambda^{\ast}$ and  $\underline\lambda_{\varepsilon}^{\ast}\to \lambda^{\ast}$
as $\varepsilon\to 0^{+}$.

Indeed, let $U(t), \bar U_{\varepsilon}(t)$ and $\underline U_{\varepsilon}(t)$ denote the $c_{0}$-semigroup generated
by $L, \bar L_{\varepsilon}$ and $\underline L _{\varepsilon}$, respectively. Obviously, for any given $T>0$,
$$
r(\underline U_{\varepsilon}(T)) \le r (U(T)) \le r (\bar U_{\varepsilon}(T)).
$$
According to \cite[Theorem 2.2]{MR4628895}, there exists $\phi\in\mathrm{Int}(X_{+})$ such that
$L\phi=\lambda^{\ast}\phi$. It follows that $U(T)\phi= r(U(T))\phi$ since $r(U(T))=e^{\lambda^{\ast}T}$.
By the variation of constants formula, we have
\[
 \left[ \underline U_{\varepsilon}(T)\phi\right] (x)\ge  \left[ U\phi\right](x)  - \varepsilon\int_{0}^{T}U(t-s) \left[ \int_{\Omega_{0}}\left[ \underline U_{\varepsilon}(s)\phi\right] (z)\,\mathrm{d}z + \left[ \underline U_{\varepsilon}(s)\phi\right] (\cdot)\right](x)\,\mathrm{d}s.
 \]
It is easy to verify that
\[
 \lim_{\varepsilon\to 0^{+}}\left\| \varepsilon\int_{0}^{T}U(t-s) \left[ \int_{\Omega_{0}}\left[ \underline U_{\varepsilon}(s)\phi\right] (z)\,\mathrm{d}z + \left[ \underline U_{\varepsilon}(s)\phi\right] \right]\,\mathrm{d}s\right\|_{X}= 0.
  \]
Hence, for any $\varepsilon^{\ast}>0$, there exists $\delta>0$ such that for all $0\le\varepsilon\le\delta$,
we have $\underline U_{\varepsilon}(T)\phi\ge \left( r(U(T))-\varepsilon^{\ast}\right) \phi$.
According to \cite[Lemma 2.4]{MR3992071}, one can deduce that
$r(\underline U_{\varepsilon}(T))\ge r(U(T))-\varepsilon^ {\ast}$ for all $ 0\le \varepsilon\le \delta$.
It follows that $\lim\limits_{\varepsilon\to 0^{+}} \underline\lambda_{\varepsilon}^{\ast}= \lambda^{\ast}$.
By the similar discussion or according to \cite[Section IX, P.497]{MR1335452}, one can show
$\bar\lambda_{\varepsilon}^{\ast}\to \lambda^{\ast}$ as $\varepsilon\to 0^{+}$.
According to Corollary \ref{c2.1}, we have the following results.

\begin{proposition}\label{pr3.1}
The following statements are valid for system \eqref{eq3.1}:
\begin{enumerate}
	\item[(1)]  If $\lambda^{\ast}\le0$, then $0$ is globally asymptotically stable in $X_{+}^{v} $.
	\item[(2)] If $\lambda^{\ast}> 0$, then system \eqref{eq3.1} admits a unique positive steady state
        $u^{\ast}$ in $\hat X_{+}^{v} $; moreover, $ u^{\ast}\in\mathrm{Int}( X_{+})$ and is globally asymptotically stable
        in $X_{+}^{v}\setminus\left\lbrace 0\right\rbrace$.
\end{enumerate}
\end{proposition}

\begin{theorem}\label{th3.1}
Assume {\rm\textbf{(B1)}} holds. For system \eqref{intro-4}, the following statements are valid:
\begin{enumerate}
	\item[(1)]  If $\lambda^{\ast}\le 0 $, then $0$ is globally asymptotically stable in $ X_{+}^{v}$.
	\item[(2)] If $\lambda^{\ast}> 0  $, then $u^{\ast}$ is globally asymptotically stable in $ X_{+}^{v}\setminus\left\lbrace 0\right\rbrace $.
\end{enumerate}
\end{theorem}

\begin{proof}
For any $\varepsilon>0$, since
$J_{\frac{1}{\rho(t)}}(z) \to J_{\frac{1}{\rho_{\infty}}}(z)$ uniformly for
$z\in S:=\left\{x-y: x\in \overline{\Omega}_{0}, y\in\overline{\Omega}_{0}\right\}$ and
$\frac{\dot{\rho}(t)}{\rho(t)} \to 0$ as $t \to +\infty$, there exists $t_{\varepsilon}>0$ such that
$$
\max\left\lbrace J_{\frac{1}{\rho_{\infty}}}(z)-\varepsilon,0\right\rbrace \le J_{\frac{1}{\rho(t)}}(z) \le
J_{\frac{1}{\rho_{\infty}}}(z)+\varepsilon, \quad \text{and } -\varepsilon\le n\frac{\dot{\rho}(t)}{\rho(t)}\le \varepsilon,
$$
for all $z\in S, t\ge t_{\varepsilon}$.
Consider the following auxiliary problems
\begin{eqnarray}\label{eq3.2}
\begin{cases}
\, \dfrac{\partial \bar u_\varepsilon}{\partial t}=  D\displaystyle\int_{\Omega_{0}}\left[ J_{\frac{1}{\rho_{\infty}}}(y-z)+\varepsilon \right]&{\hspace{-9pt}}\bar u_\varepsilon(t,z)\,\mathrm{d}z-D\bar u_\varepsilon(t,y)
\\
&\ \ \ +\varepsilon \bar u_\varepsilon(t,y) +f(\bar u_\varepsilon), \qquad t>t_{\varepsilon}, y \in \Omega_{0},
		\\[2mm]
		\, \bar u_\varepsilon(t_{\varepsilon},y)= u(t_{\varepsilon},y), &{\hspace{133pt}} y \in \Omega_{0}.
	\end{cases}
\end{eqnarray}
and
\begin{eqnarray}\label{eq3.3}
\begin{cases}
\, \dfrac{\partial \underline u_\varepsilon}{\partial t}=D\displaystyle\int_{\Omega_{0}}\left[ \max\left\lbrace  J_{\frac{1}{\rho_{\infty}}}(y-z)-\varepsilon,0\right\rbrace \right]&{\hspace{-9pt}}\underline u_\varepsilon(t,z)\,\mathrm{d}z-D\underline u_\varepsilon(t,y)
\\
&\ \  -\varepsilon \underline u_\varepsilon(t,y) +f(\underline u_\varepsilon), \quad t>t_{\varepsilon}, y \in \Omega_{0},
		\\[2mm]
		\, \underline u_\varepsilon(t_{\varepsilon},y)= u(t_{\varepsilon},y), &{\hspace{118pt}} y \in \Omega_{0}.
	\end{cases}
\end{eqnarray}
Let $\bar u_\varepsilon(t,y) $ and $\underline u_\varepsilon(t,y)$ denote the solutions of \eqref{eq3.2} and \eqref{eq3.3}, respectively.
By comparison principle, we have
\begin{equation}\label{eq3.4}
\underline u_{\varepsilon}(t,y)\le u(t,y)\le \bar u_{\varepsilon}(t,y)  \text{ for } y\in \overline{\Omega}_{0}, t\ge t_{\varepsilon}.
\end{equation}

When $\lambda^{\ast}>0$, there exists $\varepsilon_{0}>0$ such that
$\bar\lambda_{\varepsilon}^{\ast}>0$ and $\underline\lambda_{\varepsilon}^{\ast}>0$ for all $\varepsilon\le\varepsilon_{0}$.
Due to \textbf{(F)}, we can verify that $v$ is a supersolution for system \eqref{eq3.2} and \eqref{eq3.3} for all $\varepsilon>0$
small enough, and the conclusions  obtained in Section 3 still hold for these system. By Corollary \ref{c2.1},
system \eqref{eq3.2} and \eqref{eq3.3} admit a unique positive steady state $\bar u_{\varepsilon}^{\ast}(y)$ and
$\underline u_{\varepsilon}^{\ast}(y)$, respectively, and
\begin{equation}\label{eq3.5}
\lim_{t\to+ \infty} \left\| \bar u_{\varepsilon}(t,\cdot)- \bar u_{\varepsilon}^{\ast}\right\| _{X} = \lim_{t\to+ \infty}
\left\| \underline u_{\varepsilon}(t,\cdot)- \underline u_{\varepsilon}^{\ast}\right\| _{X}=0.
\end{equation}
One can verify that $\bar u_{\varepsilon}^{\ast}(y) \ge \underline u_{\varepsilon}^{\ast}(y)$, and
$\bar u_{\varepsilon}^{\ast}(y)$ is nondecreasing and $\underline u_{\varepsilon}^{\ast}(y)$ is nonincreasing
in $\varepsilon$. Hence, for each $y\in \overline{\Omega}_{0}$, $\bar u_{\varepsilon}^{\ast}(y)$ and
$\underline u_{\varepsilon}^{\ast}(y)$ converge to some $\bar u ^{\ast}(y)$ and $\underline u^{\ast}(y)$
as $\varepsilon$ tends to zero, respectively.
By Lebesgue dominated convergence theorem, it is straightforward to verify that
$\bar u ^{\ast}(y)$ and $ \underline u^{\ast}(y)$ satisfy the same equation
\[
D\int_{\Omega_{0}}J_{\frac{1}{\rho_{\infty}}}(y-z)u(z)\,\mathrm{d}z-Du(y) +f(u(y))=0.
\]
According to  Proposition \ref{pr3.1}, we have $\bar u ^{\ast}= \underline u^{\ast}= u^{\ast}\in X$.
Therefore, by Dini's theorem,
\begin{equation}\label{eq3.6}
\lim_{\varepsilon\to 0^{+}} \left\| \bar u_{\varepsilon}^{\ast}-  u^{\ast}\right\| _{X} = \lim_{\varepsilon\to 0^{+}}
\left\| \underline u_{\varepsilon}^{\ast}-  u^{\ast}\right\| _{X}=0
\end{equation}
By combining \eqref{eq3.4}--\eqref{eq3.6} with Remark \ref{pr3.2}, we obtain the global stability of $u^{\ast}$.

When $\lambda^{\ast}<0$, there exists $\varepsilon_{1}$ such that $\bar \lambda^{\ast}_{\varepsilon_{1}}<0$.
By Corollary \ref{c2.1}, one has $\bar u_{\varepsilon_{1}}(t,y)$ converges to zero uniformly in $y\in\overline{\Omega}_{0}$
as $t \to +\infty$, which implies
\begin{equation}\label{eq4.6}
	\lim\limits_{t\to+ \infty}\left\| u(t,\cdot)\right\| _{X}=0.
\end{equation}

When $\lambda^{\ast}=0 $, we have $\bar\lambda^{\ast}_{\varepsilon}>0$ for any $\varepsilon>0$.
Hence, to prove \eqref{eq4.6}, it suffices to show that $ \bar u ^{\ast}(y) \equiv 0$. According to
Proposition \ref{pr3.1} (or Lemma \ref{le2.1}), this holds true.  In view of  Remark \ref{pr3.2}, we then deduce the global asymptotic stability of $0$.
\end{proof}

By Theorem \ref{generalized-dynamical}, we also have the following
observation.
\begin{theorem}
Assume {\rm\textbf{(B1)}} holds. The following statements are valid
for system \eqref{intro-4}:
\begin{enumerate}
	\item[(1)]  If $\lambda^{\ast}\le0$, then $0$ is globally asymptotically stable in $\hat X_{+}^{v}$.
	\item[(2)] If $\lambda^{\ast}>0$, then $u^{\ast}$ is globally asymptotically stable in $\hat X_{+}^{v}\setminus\hat{\mathbf 0}$.
\end{enumerate}
\end{theorem}

\subsection{Asymptotically time-periodic domain}

In this subsection, we explore the global dynamics in the case where $\Omega_t$ is asymptotically time-periodic.
Specifically, we assume that there exists a positive $T$-periodic function $ \rho_T$ satisfying the following asymptotic condition:
\begin{itemize}
	\item[\textbf{(B2)}] $\lim\limits_{t\to +\infty}  (\rho (t) -\rho_T(t))= \lim\limits_{t\to +\infty} (\dot\rho (t)-\dot\rho_T(t))= 0$.
\end{itemize}
Consider the limiting system:
\begin{equation}\label{eq4.1}
\left\{
{\begin{array}{*{20}{l}}
	\dfrac{\partial u}{\partial t}=D\displaystyle\int_{\Omega_{0}}J_{\frac{1}{\rho_T(t)}}(y-z)u(t,z)\,\mathrm{d}z-Du(t,y)
    -n\dfrac{\dot\rho_T(t)}{\rho_T(t)}u +f(u), & t>0, y \in \Omega_{0},
        \\[4mm]
	u(0,y)=u_{0}(y),  & y \in \Omega_{0}.
	\end{array}}
\right.
\end{equation}
and the linearized system:
\begin{equation}\label{eq4.5}
\left\{
{\begin{array}{*{20}{l}}
	\dfrac{\partial u}{\partial t}=D\displaystyle\int_{\Omega_{0}}J_{\frac{1}{\rho_T(t)}}(y-z)u(t,z)\,\mathrm{d}z-Du(t,y)
    -n\dfrac{\dot\rho_T(t)}{\rho_T(t)}u +\mathcal{D}f(0)u, & t>0, y \in \Omega_{0},
        \\[4mm]
	u(0,y)=u_{0}(y),  & y \in \Omega_{0}.
	\end{array}}
\right.
\end{equation}
Let $V(t,s)$ denote the evolution family on $X$ generated by \eqref{eq4.5}.
Consider an  eigenvalue problem association with  \eqref{eq4.5}:
\begin{equation}\label{eq4.4}
\left\{
{\begin{array}{*{20}{l}}
	\dfrac{\partial u}{\partial t}=D\displaystyle\int_{\Omega_{0}}J_{\frac{1}{\rho_T(t)}}(y-z)u(t,z)\,\mathrm{d}z-Du(t,y)
    -n\dfrac{\dot\rho_T(t)}{\rho_T(t)}u +\mathcal{D}f(0)u+\lambda u, & t\in \mathbb{R}, y \in \Omega_{0},
        \\[4mm]
	u(T+t,y)=u(t, y),  &t\in \mathbb{R}, y \in \Omega_{0}.
\end{array}}
\right.
\end{equation}
According to \cite{MR3637938}, problem \eqref{eq4.4} admits the principal eigenvalue $\lambda_{T}^{\ast}$, which satisfies
$$
\lambda_{T}^{\ast}= - \omega(V) =- \frac{\ln r(V(T))}{T}.
$$
By applying Theorem \ref{th2.3}, we derive the following results.

\begin{proposition}\label{le4.1}
The following statements are valid for system \eqref{eq4.1}:
\begin{itemize}
\item[(1)]  If $\lambda_{T}^{\ast}\ge0$, then $0$ is globally asymptotically stable in $X_{+}^{v} $.
\item[(2)] If $\lambda_{T}^{\ast}<0$, then system \eqref{eq4.1} admits a unique positive steady state $u_{T}^{\ast}$ in $\tilde X_{+}^{v} $;
            moreover, $u_{T}^{\ast}\in\mathrm{Int}(\mathcal{C}_{T}^{+}) $ and is globally asymptotically stable in
		  $X_{+}^{v}\setminus\left\{0\right\}$.
	\end{itemize}
\end{proposition}

Based on Theorem \ref{th2.3} (or Lemma \ref{le2.1}), Proposition \ref{le4.1}, and the arguments in Theorem \ref{th3.1},
we have the following result.

\begin{theorem}\label{th4.1}
Assume {\rm\textbf{(B2)}} holds. For system \eqref{intro-4}, the following statements are valid:
\begin{itemize}
\item[(1)]  If $\lambda_{T}^{\ast}\ge0$, then $0$ is globally asymptotically stable in $X_{+}^{v} $.
\item[(2)]  If $\lambda_{T}^{\ast}<0$, then $u_{T}^{\ast}$ is globally asymptotically stable in
$X_{+}^{v}\setminus\left\{0\right\}$.
\end{itemize}
\end{theorem}

Theorem \ref{generalized-dynamical} gives rise to the following
observation.
\begin{theorem}
Assume {\rm\textbf{(B2)}} holds. The following statements are valid
for system \eqref{intro-4}:
\begin{enumerate}
	\item[(1)]  If $\lambda_{T}^{\ast}\ge0$, then $0$ is globally asymptotically stable in $\hat X_{+}^{v}$.
	\item[(2)]  If $\lambda_{T}^{\ast}<0$, then $u^{\ast}_{T}$ is globally asymptotically stable in $\hat X_{+}^{v}\setminus\hat{\mathbf 0}$.
\end{enumerate}
\end{theorem}

\section{Asymptotically unbounded domain}

In this section, we continue to investigate the global dynamics of system \eqref{intro-4} in the case where $\Omega_t$
is asymptotically unbounded. Accordingly, we assume that $\rho (t)$ satisfies the following infinite growth condition:
\begin{itemize}
\item [\textbf{(B3)}] $ \lim\limits_{t\to +\infty} \rho (t)= +\infty$,\,\, $\lim\limits_{t\to +\infty}
	\frac{\dot\rho(t)}{\rho (t)} =k \ge 0$ and $\Omega_{0}$ is convex.
\end{itemize}
In addition to the assumption \textbf{(J)}, we also impose the following assumption on $J$ in this section:
\begin{itemize}
\item [\textbf{(J1)}] $J$ is compactly supported with its center of mass at the origin, i.e., $\int_{\mathbb{R}^n}J(x)x\,\mathrm{d}x=0$.
\end{itemize}
Throughout this section, we always assume that \textbf{(B3)} and \textbf{(J1)} hold. Consider the limiting system:
\begin{equation}\label{eq5.1}
\left\{
{\begin{array}{*{20}{l}}
	\dfrac{\mathrm{d} w}{\mathrm{d} t}=-nkw +f(w), & t>0,
        \\[3mm]
	w(0)=w_{0}\in\mathbb{M}.
\end{array}}
\right.
\end{equation}
For convenience, denote $A:= -nkI+\mathcal{D}f(0)$ where $I$ is a identity matrix. Note that the Perron--Frobenius
theorem implies that $s(A)$ is the principal eigenvalue of $A$, that is, there exists a vector $\bar z\gg 0$ in
$\mathbb{R}^{m}$ such that $A\bar z = s(A) \bar z$. According to \cite[Theorem 2.3.4]{MR3643081}, we have the
following threshold dynamics result for system \eqref{eq5.1}.

\begin{proposition}\label{pr5.1}
The following statements are valid:
\begin{itemize}
	\item [(1)] If $s(A)\le0$, then $0$ is globally asymptotically stable for system \eqref{eq5.1} in $\mathbb{M}$.
	\item [(2)] If $s(A)>0$, then system \eqref{eq5.1} has a unique positive equilibrium $w_{e}$, and it is globally
        asymptotically stable  for system \eqref{eq5.1} in $\mathbb{M}\setminus\left\lbrace 0\right\rbrace  $.
\end{itemize}
\end{proposition}

Consider the following auxiliary problem:
\begin{equation}\label{eq5.2}
	\left\{ {\begin{array}{*{20}{l}}
			\dfrac{\mathrm{d} w}{\mathrm{d} t}=-n\dfrac{\dot\rho(t)}{\rho(t)}w +f(w), & t>0,\\[3mm]
			w(0)=w_{0}\in\mathbb{M}.\\
		\end{array}} \right.
\end{equation}
Clearly, system \eqref{eq5.1} is the limiting system of \eqref{eq5.2}. Applying the theory of asymptotically
autonomous semiflows or the theory of chain transitive sets (see \cite[Section 1.2.1]{MR3643081}), we can deduce
that if $s(A)\le 0 $, the solution $w(t;w_{0})$ of \eqref{eq5.2} satisfies
$\lim\limits_{t\to+ \infty}w(t;w_{0})=0$.

For each $\phi\in X^{v}_{+}$, let $w_{0}:=\max\limits_{y\in\overline\Omega_{0}} \phi(y)$.
According to Theorem \ref{th2.2}, we have $u(t,y;\phi)\le w(t;w_{0})$. It follows that
$\lim\limits_{t\to +\infty}\left\|u(t,\cdot;\phi) \right\|_{X} =0$.
That is, we obtain the following result.

\begin{proposition}\label{pr5.2}
If $s(A)\le 0 $, for each $\phi\in X^{v}_{+}$, the solution $u(t,y;\phi)$ of system \eqref{intro-4} satisfies
$\lim\limits_{t\to +\infty}\left\|u(t,\cdot;\phi) \right\|_{X} =0$.
\end{proposition}

In the following discussion, we focus on the global dynamics of system \eqref{intro-4} for $s(A)>0$.
Consider the following system:
\begin{equation}\label{eq5.3}
\left\{
{\begin{array}{*{20}{l}}
	\dfrac{\partial u}{\partial t}=D\displaystyle\int_{\Omega_{0}}J_{\frac{1}{\rho(t)}}(y-z)u(t,z)\,\mathrm{d}z-Du
    -n ku +f(u), & t>0, y \in \Omega_{0},
        \\[4mm]
	u(0,y)=u_{0}(y),  & y \in \Omega_{0}.
\end{array}}
\right.
\end{equation}
We first analyze the asymptotic behavior of the solutions to system \eqref{eq5.3} and then derive the dynamics of
system \eqref{intro-4} by the comparison principle. To construct a subsolution of the equation corresponding to
system \eqref{eq5.3}, we introduce an auxiliary function that partially decouples the vanishing factor
$\frac{1}{\rho(t)}$ from the effects of nonlocal diffusion.

\begin{lemma}\label{auxiliary-function}
There exist a nonnegative function $\phi\in C^1(\mathbb{R}^n)$, which is strictly positive in $\Omega_0$ and vanishes
outside it,  and a constant $T>0$ such that
\begin{equation}\label{lower-bound}
\int_{\Omega_0}J_{\frac{1}{\rho(t)}}(x-y)\phi(y)\,\mathrm{d}y-\phi(x)\ge -\dfrac{1}{\rho(t)}\phi(x), \quad \forall x\in\mathbb{R}^n,\, \,
 t\ge T.
\end{equation}
\end{lemma}

\begin{proof}
We first consider the case where $\Omega_0$ is an open ball $B_{2R}$, defined as
$B_{2R}:=\{x\in\mathbb{R}^n: |x-x_0|\le2R\}$,
where $|x|$ denotes the Euclidean distance from $x$ to the origin. Without loss of generality, assume $x_0=0$.
Define
\begin{eqnarray*}
	\varphi(x)=\begin{cases}
		\, 1, &\quad x\in B_{R},
		\\
		\, 0,&\quad x\in \mathbb{R}^n\setminus B_{2R},
		\\
		\, \dfrac{1}{R^2}(\bar{x}-x)^2,&  \quad x\in B_{2R}\setminus B_{R} ,
	\end{cases}
\end{eqnarray*}
where $\bar{x}$ is the orthogonal projection of $x$ onto $\partial B_{2R}$.
Define a smooth function
\begin{eqnarray*}
	\psi(x)=\begin{cases}
		\, C\exp\left(-\dfrac{1}{1-x^2}\right), &\quad x\in B_{1},
		\\
		\, 0,&\quad x\in \mathbb{R}^n\setminus B_{1},
	\end{cases}
\end{eqnarray*}
where $C$ is the normalizing constant such that
$\int_{\mathbb{R}^n}\psi(x)\mathrm{d}x=1$.
For $\varepsilon>0$, let $\psi_{\varepsilon}(x):=\varepsilon^{-n}\psi(x/\varepsilon)$, which defines a mollifier.
Here, we take $\varepsilon=\frac{R}{100}$ and denote $\mathrm{supp}\psi_{\varepsilon}$ by $K$.
Define
\begin{eqnarray*}
	\eta(x)=\begin{cases}
		\, 1, &\quad x\in B_{R}+K,
		\\[2mm]
		\, 0,&\quad x\in \mathbb{R}^n\setminus (B_{R}+2K),
		\\[2mm]
		\, \dfrac{1}{2}+\dfrac{1}{2}\cos\left(\pi\frac{50|x|-51R}{R}\right),&  \quad x\in (B_{R}+2K)\setminus (B_{R}+K).
	\end{cases}
\end{eqnarray*}
Clearly, $\eta\in C^1(\mathbb{R}^n)$ and satisfies $\eta(x)=1$ for $|x|\le\frac{51R}{50}$ and $\eta(x)=0$ for
$|x|\ge\frac{26R}{25}$. We define
$$
\phi(x)=\eta(x)\varphi_\varepsilon(x)+(1-\eta(x))\varphi(x), \quad x\in\mathbb{R}^n,
$$
where $\varphi_\varepsilon=\psi_{\varepsilon}\ast\varphi$.
It follows that $\phi\in C^1(\mathbb{R}^n)$ which is strictly positive in $B_{2R}$ and vanishes outside it.

Next, we verify that the function $\phi$ satisfies the inequality \eqref{lower-bound}.

(i). Let $x\in B_R+3K$.
By Taylor's Theorem, we have
\begin{align}\label{auxiliary-function-1}
    &\int_{B_{2R}}J_{\frac{1}{\rho(t)}}(x-y)\phi(y)\,\mathrm{d}y-\phi(x)    \nonumber
    \\
	&  =  \int_{B_{2R}}J_{\frac{1}{\rho(t)}}(x-y)(\phi(y)-\phi(x))\,\mathrm{d}y   \nonumber
	\\
	&  =   \int_{B_{2R}}J_{\frac{1}{\rho(t)}}(x-y)(\nabla\phi(x)(y-x)+o(|y-x|))\,\mathrm{d}y   \nonumber
	\\
	&  =   \nabla\phi(x)\int_{B_{2R}-\{x\}}J_{\frac{1}{\rho(t)}}(-z)z\,\mathrm{d}z+\int_{B_{2R}}J_{\frac{1}{\rho(t)}}(x-y)o(|y-x|)\,\mathrm{d}y   \nonumber
	\\
	&  =   \int_{B_{2R}}J_{\frac{1}{\rho(t)}}(x-y)o(|y-x|)\,\mathrm{d}y,
\end{align}
for sufficiently large $\rho(t)$. Since $\phi$ is bounded below by a positive constant on $B_R+3K$,
there exists $T_1>0$, independent of $x$, such that
\begin{equation}\label{auxiliary-function-2}
	\int_{B_{2R}}J_{\frac{1}{\rho(t)}}(x-y)\phi(y)\,\mathrm{d}y-\phi(x)
	\ge -\dfrac{1}{\rho(t)}\phi(x), \quad x\in B_R+3K,\ t\ge T_1.
\end{equation}

(ii). Let $x\in \mathbb{R}^n\setminus (B_R+3K)$.
By the definition of $\phi$, we have
\begin{eqnarray*}
	\phi(x)=\varphi(x)=\begin{cases}
		\, 0,&\quad x\in \mathbb{R}^n\setminus B_{2R},
		\\[2mm]
		\, \dfrac{1}{R^2}(\bar{x}-x)^2,&  \quad x\in B_{2R}\setminus (B_R+2K),
	\end{cases}
\end{eqnarray*}
which is convex on $\mathbb{R}^n\setminus (B_R+2K)$. Take $R(x):=2R+|x|$, then we have
$$
x=\int_{B_{R(x)}}J_{\frac{1}{\rho(t)}}(x-y)y\,\mathrm{d}y, \quad \forall x\in \mathbb{R}^n\setminus (B_R+3K),\ t\ge T_2,
$$
for sufficiently large $T_2>0$. Using Jensen's inequality \cite{H-L-P1934}, we conclude
\begin{equation}\label{auxiliary-function-3}
\phi(x)\le \int_{B_{R(x)}}J_{\frac{1}{\rho(t)}}(x-y)\phi(y)\,\mathrm{d}y
=\int_{B_{2R}}J_{\frac{1}{\rho(t)}}(x-y)\phi(y)\,\mathrm{d}y,
\end{equation}
for any $x\in \mathbb{R}^n\setminus (B_R+3K)$ and $t\ge T_2$. Thus, we have established \eqref{lower-bound} for
the case where $\Omega_0$ is an open ball.

For a general bounded convex domain $\Omega_0$ with a smooth boundary $\partial\Omega_0$, based on the John ellipsoid,
we may apply an affine transformation $\mathcal{T}$ to map $\Omega_0$ onto an open ball. Using the same approach, we
may construct a function $\phi_{ball}(\cdot)$ corresponding to the ball. By the inverse affine transformation, we
obtain a function $\phi_{\Omega_0}(\mathcal{T}^{-1}(\cdot))\in C^1(\mathbb{R}^n)$. This function is convex in a
neighborhood of $\partial\Omega_0$,  and further satisfies \eqref{lower-bound}.
This completes the proof.
\end{proof}

\begin{lemma}\label{le5.1}
Assume $s(A)>0$. Then there exist constants $\delta_{0}>0$ and $N>0$ such that for any $\delta\in\left(0,\delta_{0}\right]$,
the function
$\underline u_{\delta}(y):= \delta \bar{z}\phi(y)$,
where $\bar{z}$ is the positive eigenvector of $A$ and $\phi\in C^1(\mathbb{R}^n)$ is given by Lemma \ref{auxiliary-function},
is a subsolution of  the equations corresponding to \eqref{eq5.3} in $\left[ N,+\infty\right)\times\overline\Omega_{0}$.
\end{lemma}

\begin{proof}
We fix a real number $k_A>0$ such that $A+k_AI$ has positive diagonal entries and set $s(A):=\bar r$ and
$\epsilon_0:=\frac{\bar{r}}{2(\bar{r}+k_A)}$. By an elementary analysis (see, for example,
(6.10) in \cite{MR3420507}), there exists a vector $\hat{v}\gg 0$ in $\mathbb{R}^m$ satisfying
$$
f(u)-nku\ge Au-\epsilon_0(Au+k_Au), \quad u\in[0,\hat{v}]_{\mathbb{R}^m}.
$$
Next, we choose a sufficiently small $\delta_0>0$ such that $\underline{u}_{\delta}(x)\in[0,\hat{v}]_{\mathbb{R}^m}$
for all $x\in \overline{\Omega}_0$. It follows that for any $\delta\in(0,\delta_0]$, we have
\begin{align*}
    & D\int_{\Omega_0}J_{\frac{1}{\rho(t)}}(x-y)\underline{u}_{\delta}(y)\,\mathrm{d}y-D\underline{u}_{\delta}-nk\underline{u}_{\delta}+f(\underline{u}_{\delta})
    \\
	&  \ge  -\frac{1}{\rho(t)}D\underline{u}_{\delta}+\Big((1-\epsilon_0)A-\epsilon_0k_A I\Big)\underline{u}_{\delta}
	\\
	&  \ge  \Big(-\dfrac{\bar d}{\rho(t)}+(1-\epsilon_0)\bar{r}-\epsilon_0k_A\Big)\underline{u}_{\delta}
	\\
	&   =   \Big(-\dfrac{\bar d}{\rho(t)}+\dfrac{\bar{r}}{2}\Big)\underline{u}_{\delta},
\end{align*}
where $\bar d:=\max\left\lbrace  d_{1}, d_{2},\dots,d_{m}\right\rbrace $, and hence,
$$
D\int_{\Omega_0}J_{\frac{1}{\rho(t)}}(x-y)\underline{u}_{\delta}(y)\,\mathrm{d}y-D\underline{u}_{\delta}-nk\underline{u}_{\delta} +f(\underline{u}_{\delta})
\gg   0 \,\, \text{in} \,\, \mathbb{R}^m \text{ for all }x\in\Omega_0,
$$
provided that $\rho(t)\ge\frac{2\bar d}{\bar{r}}$.
The proof is completed.
\end{proof}

Next, we use the subsolution obtained in Lemma \ref{le5.1} to further construct a generalized subsolution (see, e.g., \cite{lam2022}) of
the equation corresponding to system \eqref{eq5.3}, which is instrumental in establishing the asymptotic behavior of solutions to the system.
Define
$\Omega_{\eta}:=\left\lbrace x\in \Omega_{0}: d(y,\partial \Omega_{0})> \eta\right\rbrace$.
It is clear that $ \Omega_{\eta}\nearrow \Omega_{0}$ as $\eta\to 0^{+}$. Let $\varrho_{\eta}(y)$ be a smooth cut-off function satisfying
$\varrho_{\eta}(y)=0 \text{ in } \Omega_{0}\setminus\Omega_{\eta}$ and $\varrho_{\eta}(y)=1$ in $\Omega_{2\eta}$.
If $s(A)>0$, by the Dancer-Hess connecting orbit theorem (see, e.g., \cite[Proposition 1]{MR1116922} and \cite[Proposition 2.1]{MR1100011}),
it follows that system \eqref{eq5.1} admits a connecting orbit $\alpha  : \mathbb{R}\to \mathbb{M}$ such that
\[
\alpha ( -\infty )= 0,\  \alpha (+\infty )= w_{e} ,\  \alpha' (t)\gg 0,  \  \forall t \in  \mathbb{R}.
\]
Since $\alpha(-\infty)=0$, there exists $\underline \tau= \underline \tau(\delta,\eta)$ such that
\begin{equation}\label{eq5.4}
	\alpha(\underline \tau)\le \underline u_{\delta}(y) \text{ for all } y \in \overline\Omega_{\eta}.
\end{equation}

\begin{lemma}\label{le5.2}
	For any given $u\in X$ with $u\mid_{\partial \Omega_{0}}=0$, then
	\[ \lim_{t\to+ \infty}\left\| \int_{\Omega_{0}}J_{\frac{1}{\rho(t)}}(\cdot-z)u(z)\,\mathrm{d}z- u\right\| _{X}=0 . \]
\end{lemma}

\begin{proof}
We extend $u(y)$ to be zero for $y\in \mathbb{R}^{n}\setminus\Omega_{0}$.
Then we have
\[
\begin{aligned}
\int_{\Omega_{0}}J_{\frac{1}{\rho(t)}}(y-z)u(z)\,\mathrm{d}z- u(y) = &\int_{\mathbb{R}^{n}} J_{\frac{1}{\rho(t)}}(y-z)\left[ u(z)-u(y)\right]\,\mathrm{d}z
\\
=& \int_{{\rm supp}J}J(z)\left[ u(y+\frac{z}{\rho(t)})-u(y)\right]\mathrm{d}z.
\end{aligned}
\]
For any $\varepsilon>0$, there exists $t_{\varepsilon}>0$, independent of $z$, such that
$$
\left\| u(\cdot+\frac{z}{\rho(t)})-u(\cdot)\right\|_{X} \le \varepsilon, \quad\forall t\ge t_{\varepsilon},\, z\in {\rm supp}J,
$$
since $\rho (t)\to +\infty $ as $t \to+ \infty$.
Consequently, we obtain
$$
\left\| \int_{{\rm supp}J}J(z)\left[ u(\cdot+\frac{z}{\rho(t)})-u(\cdot)\right]\,\mathrm{d}z\right\|_{X}\le \varepsilon,\quad\forall t\ge t_{\varepsilon}.
$$
This completes the proof.
\end{proof}

Motivated by \cite[Lemma 3.2]{MR4784487}, we have the following result.

\begin{lemma}\label{le5.3}
Assume $s(A)>0$. For any given
$\delta\in \left( 0,\delta_{0}\right], \eta>0, \underline \tau, \bar \tau\in\mathbb{R}$ satisfying $\underline \tau<\bar \tau$
and \eqref{eq5.4}, there exist two positive constants $\beta=\beta(\underline \tau , \bar \tau)$ and
$\hat N=\hat N(\eta, \underline \tau , \bar \tau)>N$ such that for any $t_{0}\ge \hat N$, the function
\[
\underline u (t,y) := \max\left\lbrace \underline u_{\delta}(y), \alpha(\ln\varrho_{\eta}(y)+\bar \tau -(\bar \tau-\underline \tau)e^{-\beta(t-t_{0})}) \right\rbrace
\]
is a generalized subsolution of  the equations corresponding to \eqref{eq5.3} in $ \left[ t_{0},+\infty\right)\times \overline \Omega_{0} $.
\end{lemma}

\begin{proof}
For any given $\eta>0$ and $\bar \tau >\underline \tau$, we define that
\begin{align*}
	c_{1}&:= c_{1}(\underline \tau, \bar \tau)=\max_{i=1,\dots,m}\left\lbrace \max_{s\in[\underline \tau, \bar \tau]}\dfrac{\alpha_{i}'(s)}{\alpha_{i}(s)}\right\rbrace >0,
	\\[2mm]
	c_{2}&:= c_{2}(\underline \tau, \bar \tau)=\min_{i=1,\dots,m}\left\lbrace \min_{s\in[\underline \tau, \bar \tau]}\dfrac{\alpha_{i}'(s)}{2\alpha_{i}(s)}\right\rbrace >0.
\end{align*}
We extend the function $\alpha(\ln\varrho_{\eta}(y)+\tau)$ to  $\overline\Omega_{0}$, defining it as $0$ at any point
$y\in\overline\Omega_{0} $ where $\varrho_{\eta}(y)=0 $, for all $\tau \in [\underline \tau, \bar \tau]$.
Clearly, under this extension, $\alpha(\ln\varrho_{\eta}(y)+\tau) $ is a continuous function on
$\overline\Omega_{0}\times [\underline \tau, \bar \tau ]$, satisfying $\alpha(\ln\varrho_{\eta}(y)+\tau)=0 $ for
$\partial \Omega_{0}\times [\underline \tau, \bar \tau ]$. According to Lemma \ref{le5.2}, there exists
$\hat N=\hat N (\eta, \underline \tau, \bar \tau)>N$ such that
\begin{equation}\label{eq5.5}
D\int_{\Omega_{0}} J_{\frac{1}{\rho(t)}}(y-z)\alpha(\ln\varrho_{\eta}(z)+\tau)\,\mathrm{d}z-D \alpha(\ln\varrho_{\eta}(y)+\tau)
\ge -\dfrac{1}{2}\min_{s\in[\underline \tau, \bar \tau]}\alpha'(s).
\end{equation}
Let $\beta:= \frac{c_{2}}{c_{1}(\bar \tau- \underline \tau)}$ and $\tau(t):=\bar \tau - (\bar \tau- \underline \tau)e^{-\beta(t-t_{0})}$,
where $t_{0}\ge \hat N$ is an arbitrarily chosen constant. Define the set
$$
\mathcal{D}:= \left\lbrace (t,y)\in \left[ t_{0}, +\infty \right)\times \overline \Omega_{0} : \alpha(\ln \varrho_{\eta}(y)+\tau (t))\ge \underline u_{\delta}(y)\right\rbrace.
$$
Clearly, it holds that
$\underline u(t,y)= \alpha(\ln \varrho_{\eta}(y)+\tau (t))$ in $\mathcal{D}$.
By \eqref{eq5.4}, and due to the monotonicity of $\alpha$, we have
$\ln \varrho_{\eta}(y)+\tau (t)\ge \underline \tau$ in $\mathcal{D}$.
At the same time, direct computation yields
\begin{align}\label{eq5.6}
	\dfrac{\partial}{\partial t} \alpha(\ln \varrho_{\eta}(y)+\tau(t))
	&  \le  c_{1}(\bar \tau- \underline \tau)\beta \alpha(\ln \varrho_{\eta}(y)+\tau (t))   \nonumber
	\\
	&  =   c_{2}\alpha(\ln \varrho_{\eta}(y)+\tau (t)) \  \text{ in } \mathcal{D},
\end{align}
and
\begin{align}\label{eq5.7}
	f(\alpha(\ln \varrho_{\eta}(y)+\tau (t))-nk\alpha(\ln \varrho_{\eta}(y)+\tau (t) )
	&  =   \alpha'(\ln \varrho_{\eta}(y)+\tau (t))   \nonumber
	\\
	&  \ge  \dfrac{1}{2}\min_{s\in[\underline \tau, \bar \tau]}\alpha'(s)+ c_{2}\alpha(\ln \varrho_{\eta}(y)+\tau (t)) \ \text{ in } \mathcal{D}.
\end{align}
Together with \eqref{eq5.5}--\eqref{eq5.7}, and using the fact that
$$
D\int_{\Omega_{0}} J_{\frac{1}{\rho(t)}}(y-z)\underline u(t,z)\,\mathrm{d}z \ge D\int_{\Omega_{0}} J_{\frac{1}{\rho(t)}}(y-z)\alpha(\ln\varrho_{\eta}(z)+\tau(t))\,\mathrm{d}z,
$$
we obtain that
\begin{equation}\label{eq5.8}
	\dfrac{\partial }{\partial t}\alpha(\ln \varrho_{\eta}(y)+\tau(t)) \le D\int_{\Omega_{0}} J_{\frac{1}{\rho(t)}}(y-z)\underline u(t,z)\,\mathrm{d}z- D \underline u -nk \underline u +f(\underline u)\ \text{ in } \mathcal{D}.
\end{equation}
By virtue of Lemma \ref{le5.1}, in the set
$$
\mathcal{D}^{c}:= \left\lbrace (t,y)\in \left[ t_{0}, +\infty \right)\times \overline \Omega_{0} : \underline u_{\delta}(y) >\alpha(\ln \varrho_{\eta}(y)+\tau (t))\right\rbrace,
$$
we have
\begin{equation}\label{eq5.9}
0=\dfrac{\partial \underline u_{\delta}(y) }{\partial t}\le D\int_{\Omega_{0}} J_{\frac{1}{\rho(t)}}(y-z)\underline u(t,z)\,\mathrm{d}z - D \underline u -nk \underline u +f(\underline u).
\end{equation}
Combining \eqref{eq5.8} and \eqref{eq5.9}, it follows that $\underline u(t,y)$ is a generalized subsolution of  the equations corresponding to system \eqref{eq5.3} in $ \left[ t_{0},+\infty\right)\times \overline \Omega_{0} $.
\end{proof}

Next, using the subsolution constructed in Lemmas \ref{le5.1} and \ref{le5.3}, we establish the asymptotic behavior
of solutions to system \eqref{eq5.3}. Furthermore, under assumption \textbf{(B3)}, we perturb system \eqref{eq5.3}
and employ the comparison principle to analyze the asymptotic behavior of solutions to system \eqref{intro-4}.

\begin{proposition}\label{th5.1}
Assume $s(A)>0$. For any $\phi\in X_{+}^{v}\setminus\{0\}$, the solution $u(t,y;\phi)$ of system \eqref{eq5.3}
satisfies $\lim\limits_{t\to+ \infty} u(t,y;\phi)= w_{e}$ uniformly in any compact subset of $\Omega_{0}$.
\end{proposition}

\begin{proof}
According to Theorem \ref{th2.2}, the solution $u(t,\cdot;\phi)\in\mathrm{Int}(X_{+}^{v})$ for $t>0$. We can choose
$\delta \in \left( 0, \delta_{0}\right] $ small enough such that $u(1,y;\phi)\ge \underline u_{\delta}(y)$ for all
$y\in \overline \Omega_{0}$. It follows that
$u(t,y;\phi)\ge \underline u_{\delta}(y) $ in $\left[ 1,+\infty\right) \times \overline \Omega_{0}$.

Let $\eta>0$ be an arbitrarily given constant,  $\underline \tau $ be chosen to satisfy \eqref{eq5.4}, and $\bar \tau>\underline \tau$
be another arbitrarily given constant. Set $t_{0}=\max\left\lbrace 1,\hat N\right\rbrace $,  and let $\underline u$ be
defined as in Lemma \ref{le5.3} with time $t_{0}$. Clearly, $\underline u(t_{0},y)= \underline u_{\delta}(y)$
on $\overline\Omega_{0} $ due to \eqref{eq5.4}. According to Theorem \ref{th2.2}, we have
$u(t,y;\phi)\ge \underline u(t,y) $ in $\left[ t_{0},+\infty\right) \times \overline \Omega_{0}$.
It follows that
\[
u(t,y;\phi)\ge \alpha(\ln \varrho_{\eta}(y)+\tau (t)) =\alpha(\tau (t)), \ \forall y\in \Omega_{2\eta}, t\in \left[ t_{0},+\infty\right),
\]
since $\varrho_{\eta}(y)=1 $ in $\Omega_{2\eta}$. Due to $\lim_{t\to+ \infty}\tau (t)=\bar \tau$, we have
\[
\liminf_{t \to +\infty} u(t,y;\phi) \ge \alpha(\bar \tau) \text{ uniformly  for } y \in \Omega_{2\eta}.
\]
Letting $\bar \tau \to +\infty$, then
\begin{equation}\label{eq5.10}
\liminf_{t \to +\infty} u(t,y;\phi) \ge w_{e} \text{ uniformly  for } y \in \Omega_{2\eta}.
\end{equation}
On the other hand, let $w_{0}=\max\limits_{y\in \overline\Omega_{0}}\phi(y)$.
By comparison principle, we have $u(t,y;\phi)\le w(t;w_{0})$ for $t\ge 0$ and $y\in \overline \Omega_{0}$,
where $w(t;w_{0})$ is the solution of system \eqref{eq5.1} with $w(0;w_{0})=w_{0}$. By Proposition \ref{pr5.1},
we have
\begin{equation}\label{eq5.11}
	\limsup_{t\to +\infty} u(t,y;\phi)\le w_{e} \text{ uniformly  for } y \in \Omega_{0}.
\end{equation}
Together with \eqref{eq5.10}--\eqref{eq5.11}, we deduce
$\lim\limits_{t\to+ \infty} u(t,y;\phi)=w_{e}$
uniformly for $ y \in \Omega_{2\eta}$.
By the arbitrariness of $\eta$, it follows that the conclusion holds for any compact subset of $\Omega_{0}$.
\end{proof}

\begin{proposition}\label{c5.1}
Assume $s(A)>0$. For any $\phi\in X_{+}^{v}\setminus\{0\}$, the solution $u(t,y;\phi)$ of
system \eqref{intro-4} satisfies $\lim\limits_{t\to+ \infty} u(t,y;\phi)= w_{e}$ uniformly in any compact
subset of $\Omega_{0}$.
\end{proposition}

\begin{proof}
Since $s(A)>0$, there exists $\varepsilon_{0}>0$ such that $s(A\pm\varepsilon I)>0$ for all $\varepsilon\in (0,\varepsilon_{0}]$.
By assumption \textbf{(B3)}, there exists $t_{\varepsilon}$ such that
$$
nk-\varepsilon\le n\dfrac{\dot\rho(t)}{\rho(t)}\le nk+\varepsilon \ \textrm{ for all }t\ge t_{\varepsilon}.
$$
Let $u_{\pm \varepsilon}^{\ast}$ denote the positive equilibrium of the system
$$
u'(t)= -nku \pm\varepsilon u +f(u), \quad u\in\mathbb{M}.
$$
It is straightforward to verify that $\lim_{\varepsilon\to 0^{+}} u_{\pm \varepsilon}^{\ast}=w_{e}$.
Notice that every nontrivial solution of
$u'(t)= -nku +\varepsilon u +f(u)$
is a supersolution for \eqref{intro-4} in
$\left[ t_{\varepsilon}, +\infty\right) \times \overline\Omega_{0}$.
By Theorem \ref{th2.2} and Proposition \ref{pr5.1}, it follows that
$$
\limsup_{t\to +\infty}u(t,y;\phi)\le u_{+\varepsilon}^{\ast}\ \textrm{ uniformly for } y\in \overline \Omega_{0}.
$$
Letting $\varepsilon\to 0^{+}$, we obtain
\begin{equation}\label{eq5.12}
	\limsup_{t\to +\infty}u(t,y;\phi)\le w_{e} \text{ uniformly for } y\in \overline \Omega_{0}.
\end{equation}
On the other hand, consider the following auxiliary problem:
\[
\left\{
{\begin{array}{*{20}{l}}
	\dfrac{\partial u}{\partial t}=D\displaystyle\int_{\Omega_{0}}J_{\frac{1}{\rho(t)}}(y-z)u(t,z)\,\mathrm{d}z-Dv
    -n ku -\varepsilon u+f(u), & t>t_{\varepsilon}, y \in \Omega_{0},
        \\[4mm]
	u(t_{\varepsilon},y)=u(t_{\varepsilon},y;\phi),  & y \in \Omega_{0},
\end{array}}
\right.
\]
and denote its solution by $\psi$.
By Theorem \ref{th2.2}, it holds that
$u(t,y;\phi)\ge \psi(t,y)$ for $t\ge t_{\varepsilon}$ and $y \in \overline \Omega_{0}$.
Furthermore, by Proposition \ref{th5.1}, we have
$\lim\limits_{t\to+ \infty} \left\| \psi(t,\cdot)- u_{-\varepsilon}^{\ast}\right\| _{C(\Omega')}= 0$,
for any compact subset $\Omega'\subset\Omega_0$.
It follows that
\begin{equation}\label{eq5.13}
\liminf_{t \to +\infty} u(t,y;\phi) \ge w_{e}
\end{equation}
uniformly for $y$ in any compact subset of $\Omega_{0}$.
Combining \eqref{eq5.12} and \eqref{eq5.13}, the proof is complete.
\end{proof}

By combining Propositions \ref{pr5.2} and  \ref{c5.1}, we
have the following  threshold-type results on the global dynamics of system \eqref{intro-4}.

\begin{theorem}\label{th5.2}
	The following statements are valid for system \eqref{intro-4}:
\begin{itemize}
	\item [(1)] If $s(A)\le0$, then $0$ is globally asymptotically stable in $X_{+}^{v}$.
	\item [(2)] If $s(A)>0$, for each $u_{0}\in X_{+}^{v}\setminus\{0\}$,  the solution $u(t,y;u_{0})$ of system \eqref{intro-4} satisfies $\lim\limits_{t\to+ \infty} u(t,y,u_{0})=w_{e}$ uniformly in any compact subset of $\Omega_{0}$, where $w_{e}$ is shown as in Proposition \ref{pr5.1}(2).
\end{itemize}
\end{theorem}

As a consequence of Theorems \ref{th2.2} and \ref{generalized-dynamical}, we also have the following result.
\begin{theorem}
For system \eqref{intro-4}, the following statements are valid:
\begin{enumerate}
	\item[(1)]  $s(A)\le0$, then $0$ is globally asymptotically stable in $\hat X_{+}^{v}$.
	\item[(2)]  If $s(A)>0$, for each $u_{0}\in \hat X_{+}^{v}\setminus\hat{\mathbf 0}$, the solution $u(t,y;u_{0})$ of system \eqref{intro-4} satisfies $\lim\limits_{t\to+ \infty} u(t,y,u_{0})=w_{e}$ uniformly in any compact subset of $\Omega_{0}$.
\end{enumerate}
\end{theorem}

\section{Numerical simulations}

In this section, we present numerical simulations based on a practical model to illustrate our analytic results.
Specifically, we consider a West Nile (WN) virus model proposed in \cite{Wonham,MR2224757}, which is described by the following ODE dynamical system:

\begin{equation}\label{eq6.1}
\begin{aligned}
	\frac{\mathrm{d} L_V}{\mathrm{d} t} &= b_V (S_V + E_V + I_V) - m_V L_V - d_L L_V, \\
	\frac{\mathrm{d} S_V}{\mathrm{d} t} &= -\alpha_V \beta_R \frac{I_R}{N_R} S_V + m_V L_V - d_V S_V, \\
	\frac{\mathrm{d} E_V}{\mathrm{d} t} &= \alpha_V \beta_R \frac{I_R}{N_R} S_V - (\kappa_V + d_V) E_V, \\
	\frac{\mathrm{d} I_V}{\mathrm{d} t} &= \kappa_V E_V - d_V I_V,\\
	\frac{\mathrm{d} S_R}{\mathrm{d} t} &= -\alpha_R \beta_R \frac{S_R}{N_R} I_V + \eta_R R_R, \\
	\frac{\mathrm{d} I_R}{\mathrm{d} t} &= \alpha_R \beta_R \frac{S_R}{N_R} I_V - (\delta_R + \gamma_R) I_R, \\
	\frac{\mathrm{d} R_R}{\mathrm{d} t} &= \gamma_R I_R - \eta_R R_R , \\
	\frac{\mathrm{d} X_R}{\mathrm{d} t} &= \delta_R I_R.
\end{aligned}
\end{equation}
The meanings of the parameters in the system \eqref{eq6.1} are provided in Table \ref{table:variables}.

\begin{table}[h!]
	\centering
	\caption{Description of variables and parameters of the model \eqref{eq6.1}.} 
	\label{table:variables} \vspace{5pt}
	\small
	\begin{tabular}{c p{10cm}} 
		\toprule
		\textbf{Parameters} & \textbf{Description} \\
		\midrule
		$L_V$, $S_V$, $E_V$, $I_V$ & the numbers of the  classes of larval, susceptible, exposed and infectious (infective) adult in female mosquito, respectively  \\
		$S_R$, $I_R$, $R_R$, $X_R$ & the numbers of the  classes of  susceptible, infectious, removed and dead in bird, respectively \\
		$N_R = S_R + I_R + R_R$    & total live population of bird\\
		$b_V$                 & mosquito birth rate\\
		$d_L, d_V$            & larval, adult mosquito death rate\\
		$\delta_R$            & bird death rate, caused by virus\\
		$\alpha_V, \alpha_R$  & WN transmission probability per bite to mosquitoes, birds\\
		$\beta_R$             &  biting rate of mosquitoes on birds\\
		$m_V$                 & mosquito maturation rate\\
		$\kappa_V$            & virus incubation rate in mosquitoes\\
		$\gamma_R$            & bird recovery rate from WN\\
		$\eta_R$              & bird loss of immunity rate\\
		\bottomrule
	\end{tabular}
\end{table}

As noted by Lewis et al. \cite{MR2224757}, \textit{``both bird and mosquito movements actually involve a mixture of local interactions, long-distance
 dispersal and, in the case of birds, migratory flights"}. Thus, it is natural to incorporate nonlocal diffusion into system \eqref{eq6.1}.
Under appropriate assumptions (\cite{MR2224757}), this system simplifies to:

\begin{equation}\label{eq6.2}
	\begin{cases}
		{\begin{aligned}
				\dfrac{\partial I_{V}}{\partial t} =  D_V & \int_{\Omega_0}J(x-z) I_{V}(z,t)\,\mathrm{d}z -D_V I_{V}
				\\& +\alpha_{V}\beta_R \frac{  I_{R}}{N_R } \left( A_{V} - I_{V} \right) -  d_{V}  I_{V},
		\end{aligned} }& x \in \Omega_0, \ t > 0, \\
		{\begin{aligned}
				\dfrac{\partial I_{R}}{\partial t}  =  D_{R} &\int_{\Omega_0}J(x-z)I_{R}(z,t)\,\mathrm{d}z-D_{R}I_{R}\\& +\alpha_{R}\beta_R \frac{I_V }{N_R} \left( N_{R} - I_{R}\right)   - \gamma_R  I_{R},
		\end{aligned}}  & x\in \Omega_0, \ t > 0,
	\end{cases}
\end{equation}
where $D_V $ and $D_{R} $ denote the diffusion coefficients with $D_V\ll D_{R}$ as mosquitoes disperse more slowly than birds, and $A_{V}:= S_V + E_V + I_V$ and $N_R$ are constants.
We now consider the nonlocal system \eqref{eq6.2} on a time-varying domain and use the approach discussed in Section 2 to rewrite \eqref{eq6.2} in the following form. For convenience, the unknown functions retain their original symbols.
\begin{equation}\label{eq6.3}
	\begin{cases}
		{\begin{aligned}
				\dfrac{\partial I_{V}}{\partial t} =  D_V &\int_{\Omega_0}\rho(t)J(\rho(t)(y-z)) I_{V}(z,t)\,\mathrm{d}z -D_V I_{V}
				\\& +\alpha_{V}\beta_R \frac{  I_{R}}{N_R } \left( A_{V} - I_{V} \right) - \left( d_{V} + \frac{ \dot\rho(t)}{\rho(t)} \right) I_{V},
		\end{aligned} }& y \in \Omega_0, \ t > 0, \\
		{\begin{aligned}
				\dfrac{\partial I_{R}}{\partial t}  =  D_{R} &\int_{\Omega_0}\rho(t)J(\rho(t)(y-z)) I_{R}(z,t)\,\mathrm{d}z-D_{R}I_{R}\\& +\alpha_{R}\beta_R \frac{I_V }{N_R} \left( N_{R} - I_{R}\right)   - \left( \gamma_R + \frac{ \dot\rho(t)}{\rho(t)} \right) I_{R},
		\end{aligned}}  & y \in \Omega_0, \ t > 0.
	\end{cases}
\end{equation}
In this system, we have $\bar v:= (A_{V}, N_{R})$ and $\mathbb{M}:= [0,A_{V}]\times [0,N_{R}]$.
It follows that all the above analytical results are applicable to \eqref{eq6.3}.
For simplicity, we set $\Omega_{0}:=(0,1)$ and define $J(x)$ as follows:
\[J(x):=
\begin{cases}
1-\left| x\right| , & \left| x\right|\le 1,\\
0, & \left| x\right|> 1.
\end{cases} \]
At the following numerical simulations, we choose the initial values
\[ I_V(0,y) = \sin \dfrac{y}{2};  \quad     I_R(0,y) = 0.5 y-0.5y^{2}.\]
For illustrative purpose, we set
\[
\gamma_R=0.01, \quad N_{R}=100, \quad A_{V}=5000, \quad  D_{V}=0.03,\quad  D_{R}= 2 .
\]
For other parameters, we use values as estimated in \cite{Wonham}, namely,
\[
d_{V}=0.029, \quad \alpha_{V}=0.24, \quad \alpha_{R}=1, \quad \beta_R=0.16.
 \]

Next, we consider different asymptotic cases of $ \rho $ and conduct numerical simulations to validate
the theoretical results established for both asymptotically bounded and unbounded cases.

\vspace{0.2cm}
\noindent \textbf{Case 1. Asymptotically fixed domain}

Take $\rho(t):=\frac{\frac{3}{2}e^{2 t}}{1+\frac{1}{2}e^{2 t}}, t\geq 0 $.  It follows that $\lim\limits_{t\to+ \infty}\rho(t)=3$ and $\lim\limits_{t\to+ \infty} \dot\rho(t)=0$.
Let $\lambda^{\ast}$ be the principal eigenvalue of the following eigenvalue problem
\[
\begin{cases}
\displaystyle	D_V \int_{\Omega_0}\rho_{\infty}J(\rho_{\infty}(y-z)) I_{V}(z,t)\,\mathrm{d}z -D_V I_{V}
		+ \frac{ \alpha_{V}\beta_R A_{V} }{N_R } I_{R} - d_{V}  I_{V}=\lambda^{\ast}I_{V} ,
 \\[3mm]
\displaystyle	D_{R} \int_{\Omega_0}\rho_{\infty}J(\rho_{\infty}(y-z))I_{R}(z,t)\,\mathrm{d}z-D_{R}I_{R} +\alpha_{R}\beta_R I_V   -  \gamma_R  I_{R} = \lambda^{\ast}I_{R}.
\end{cases}
\]
It can be shown that $\lambda^{\ast}>0$, and numerical computations yield $\lambda^{\ast}\approx 0.2196$.

\vspace{0.2cm}
\noindent \textbf{Case 2. Asymptotically periodic domain}

Take $\rho(t):=(1-e^{-t})\left( 0.2\sin(0.1\pi t) +1\right) +1 $. It follows that
$\lim\limits_{t\to +\infty}  (\rho (t) -\rho_T(t))= \lim\limits_{t\to +\infty} (\dot\rho (t)-\dot\rho_T(t))= 0$, where $T=20$
and $\rho_{T}(t)= 0.2\sin( 0.1\pi t) +1$.
Let $\lambda^{\ast}_{T}$ be the principal eigenvalue of the following eigenvalue problem
\[
\begin{cases}
\begin{aligned}
	\displaystyle \dfrac{\partial I_{V}}{\partial t} =	 D_V & \int_{\Omega_0}\rho_{T}(t)J(\rho_{T}(t)(y-z)) I_{V}(z,t)\,\mathrm{d}z -D_V I_{V}\\
	& + \frac{ \alpha_{V}\beta_R A_{V} }{N_R } I_{R} - d_{V}  I_{V}- \frac{\dot\rho_{T}(t)}{\rho_{T}(t)}I_{V}+\lambda^{\ast}_{T}I_{V},
\end{aligned} &t\in \mathbb{R}, y\in (0,1),
	\\[3mm]
\begin{aligned}
	\displaystyle \dfrac{\partial I_{R}}{\partial t} = D_{R}  & \int_{\Omega_0}\rho_{T}(t)J(\rho_{T}(t)(y-z))I_{R}(z,t)\,\mathrm{d}z-D_{R}I_{R}\\ &+\alpha_{R}\beta_R I_V   -  \gamma_R  I_{R} - \frac{\dot\rho_{T}(t)}{\rho_{T}(t)}I_{R}+ \lambda^{\ast}_{T}I_{R},
\end{aligned}	 &t\in \mathbb{R}, y\in (0,1),\\
	I_{V}(T+t,y)=I_{V}(t,y), I_{R}(T+t,y)=I_{R}(t,y), &t\in \mathbb{R}, y\in (0,1).
\end{cases}
\]
It can be shown that $\lambda^{\ast}_{T}<0$, and numerical computations yield $\lambda^{\ast}_T\approx -0.2829$.

\vspace{0.2cm}
\noindent \textbf{Case 3. Asymptotically unbounded domain}

Let $\rho(t):= 0.5t+1$. Then $k:=\lim\limits_{t\to +\infty}\frac{\dot\rho(t)}{ \rho(t)}=0 $.
We choose this linear function because, in numerical simulations, a rapidly growing \( \rho \) can lead to increased numerical errors and reduced simulation accuracy.
Denote
\[ A:= \left( \begin{array}{ll}
	-d_{V}-k & \frac{ \alpha_{V}\beta_R A_{V} }{N_R }\\
	\alpha_{R}\beta_R & -\gamma_R-k
\end{array}\right) \]
By calculation, $s(A)=\frac{-0.039+\sqrt{0.039^{2}+1.22764}}{2}-k\approx 0.5348>0 $.

\vspace{0.2cm}
\noindent\textbf{Case 4. Asymptotically fixed domain with an asymmetric kernel}

We are also interested in how an asymmetric kernel affects the solution distribution.
Since the kernel is not necessarily symmetric or compactly supported in the asymptotically fixed domain case,
we consider the following kernel $J$:
\[
J(x) =
\begin{cases}
	C e^{-\frac{x^2}{2}}, & x \le  0, \\
	C(1-2x), & 0<x<\frac{1}{2},\\
	0, & x\ge \frac{1}{2},
\end{cases}
\]
where $C= 1/( 0.25+ \sqrt{\pi/2})$ is the normalizing constant. A direct computation shows that
$\int_{\mathbb{R}}J(x)\,\mathrm{d}x=1$ and $\int_{\mathbb{R}}J(x)x\,\mathrm{d}x\approx-0.6375<0$.
Other parameters are the same as in Case 1. Through similar calculations, we obtain $ \lambda^{\ast}_{1}\approx 0.1181$.

From Fig. \ref{fig:fixed}--\ref{fig:unbounded}, we observe that the results of the numerical simulations are consistent with our theoretical proofs. Although we have chosen relatively small diffusion coefficients, the influence of spatial diffusion is still evident in the figures. Furthermore, a symmetrical pattern emerges as time evolves, characterized by a central peak and decreasing values on both sides of the graph. This results from the symmetry of the kernel and the homogeneous Dirichlet boundary conditions. Additionally, as shown in Fig. \ref{fig:asymmetric}, the distribution pattern of the kernel function significantly affects the location of population aggregation. Since kernel's center of mass lies on the negative half-axis, the populations exhibit a significant leftward shift, leading to the aggregation (or maximum) point appearing near the leftmost boundary.

Moreover, from Fig. \ref{fig:unbounded}, it is shown that the solution approaches a constant in the middle part, while near the boundary, it decreases rapidly within a small neighborhood of the boundary as time evolves.
This phenomenon primarily arises from the shrinking support of the kernel over time, as well as the homogeneous Dirichlet boundary condition. Furthermore, it highlights the necessity of imposing the conditions on compact subsets, as demonstrated in our proof in Section 5.

\captionsetup{labelsep=period}
\renewcommand{\figurename}{Fig.}
\begin{figure}[htbp]
	\centering
	\subfigure[]{
		\begin{minipage}{7cm}
			\centering
			\includegraphics[scale=0.5]{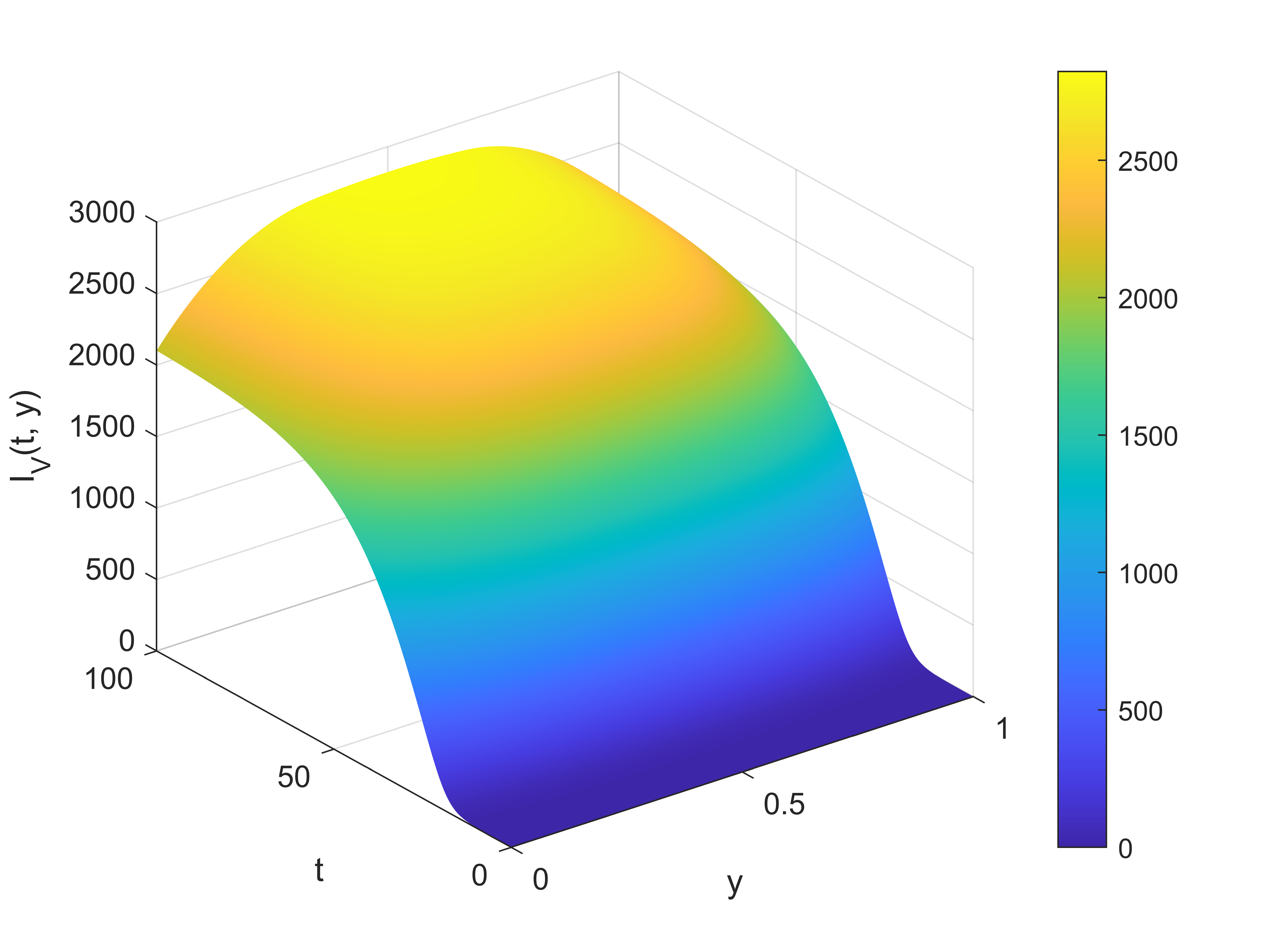}
		\end{minipage}
	}
	\hspace{0.02\textwidth}
	\subfigure[]{
		\begin{minipage}{7cm}
			\centering
			\includegraphics[scale=0.5]{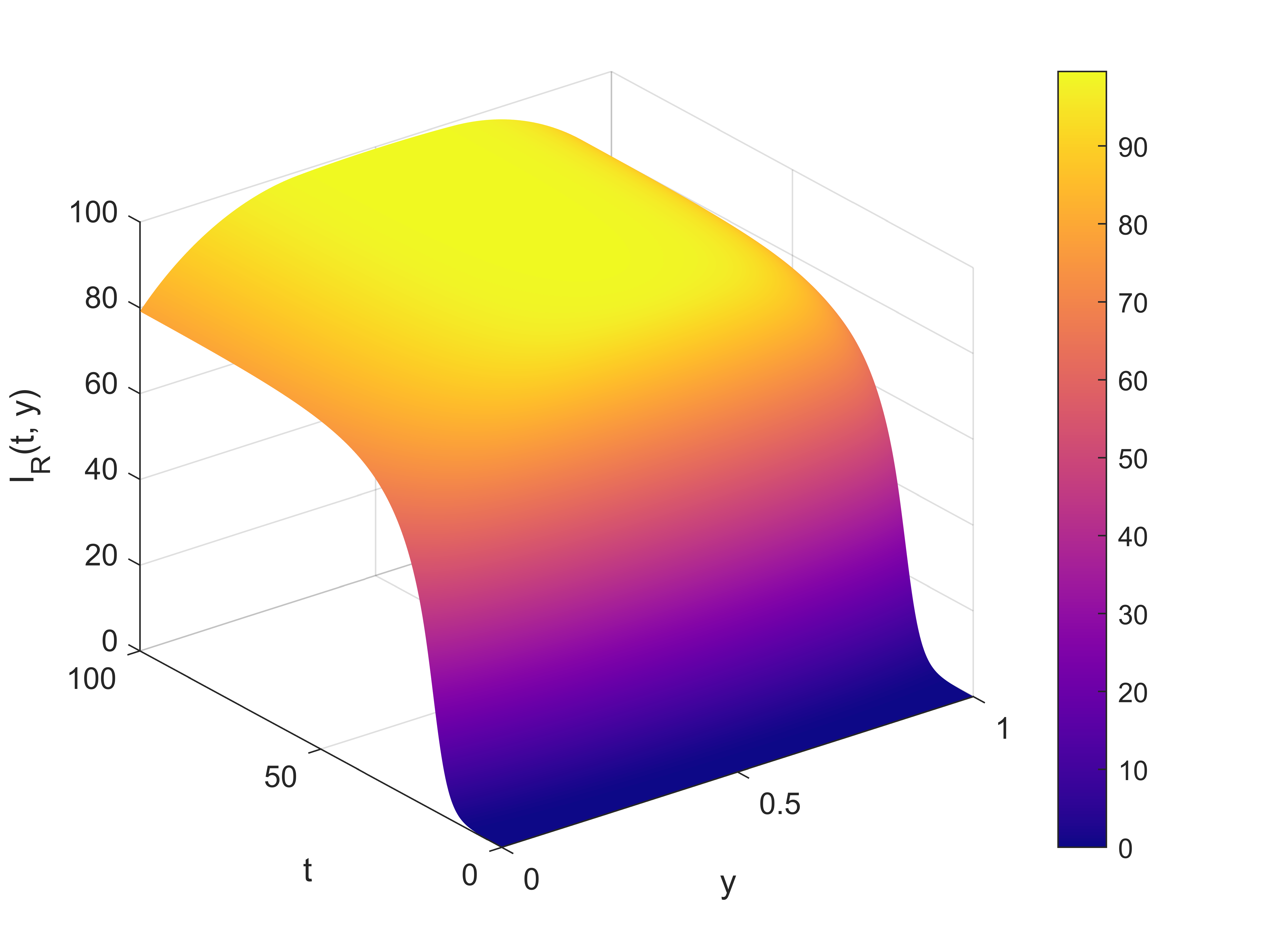}

		\end{minipage}
	}
	\caption{Dynamics of $I_V$ and $I_R $ in Case 1.}
	\label{fig:fixed}
\end{figure}

\begin{figure}[htbp]
	\centering
	\subfigure[]{
		\begin{minipage}{7cm}
			\centering
			\includegraphics[scale=0.5]{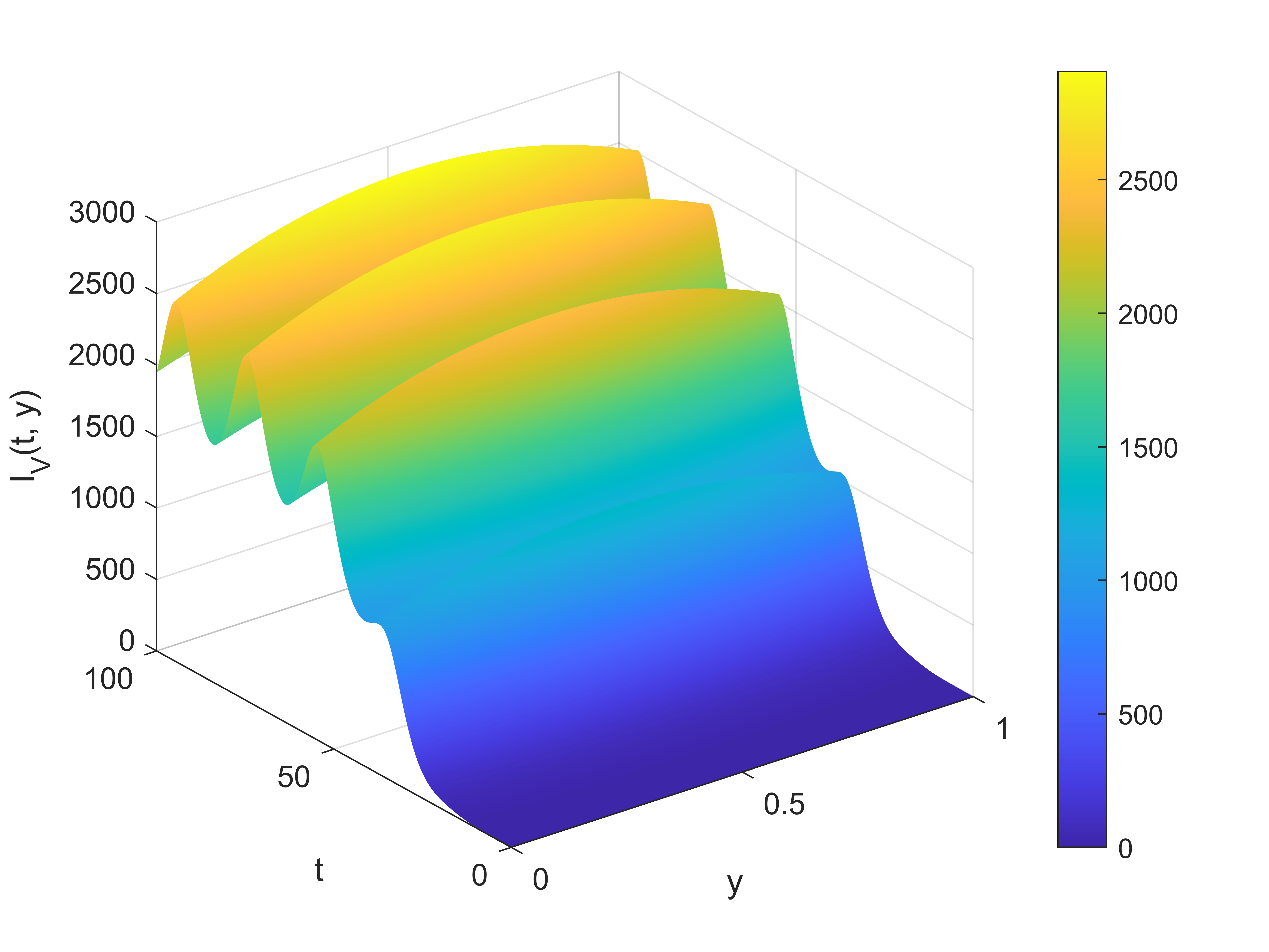}
		\end{minipage}
	}
	\hspace{0.02\textwidth}
	\subfigure[]{
		\begin{minipage}{7cm}
			\centering
			\includegraphics[scale=0.5]{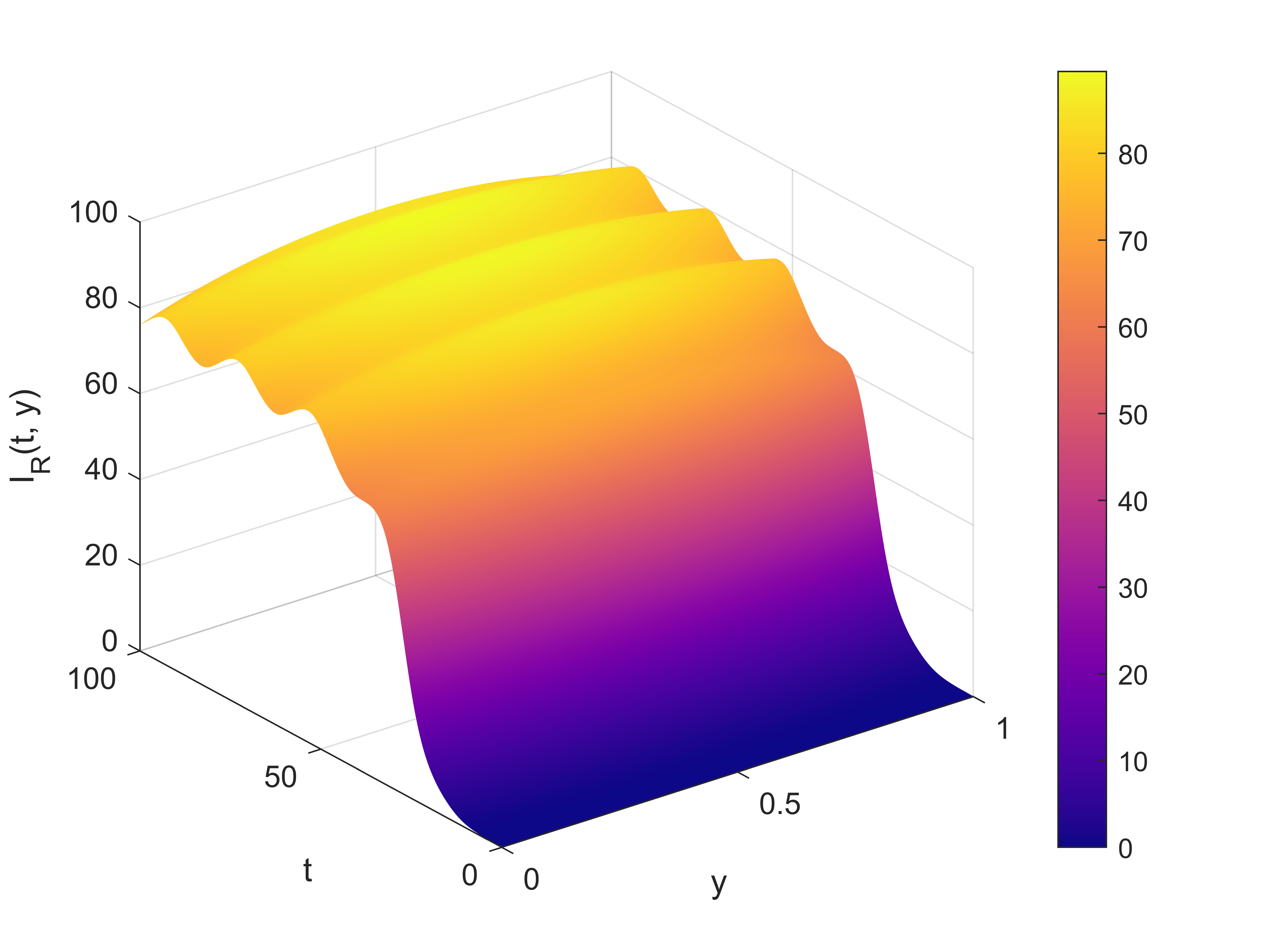}

		\end{minipage}
	}
	\caption{Dynamics of $I_V$ and $I_R $ in Case 2.}
	\label{fig:periodic}
\end{figure}

\begin{figure}[htbp]
	\centering
	\subfigure[]{
		\begin{minipage}{7cm}
			\centering
			\includegraphics[scale=0.5]{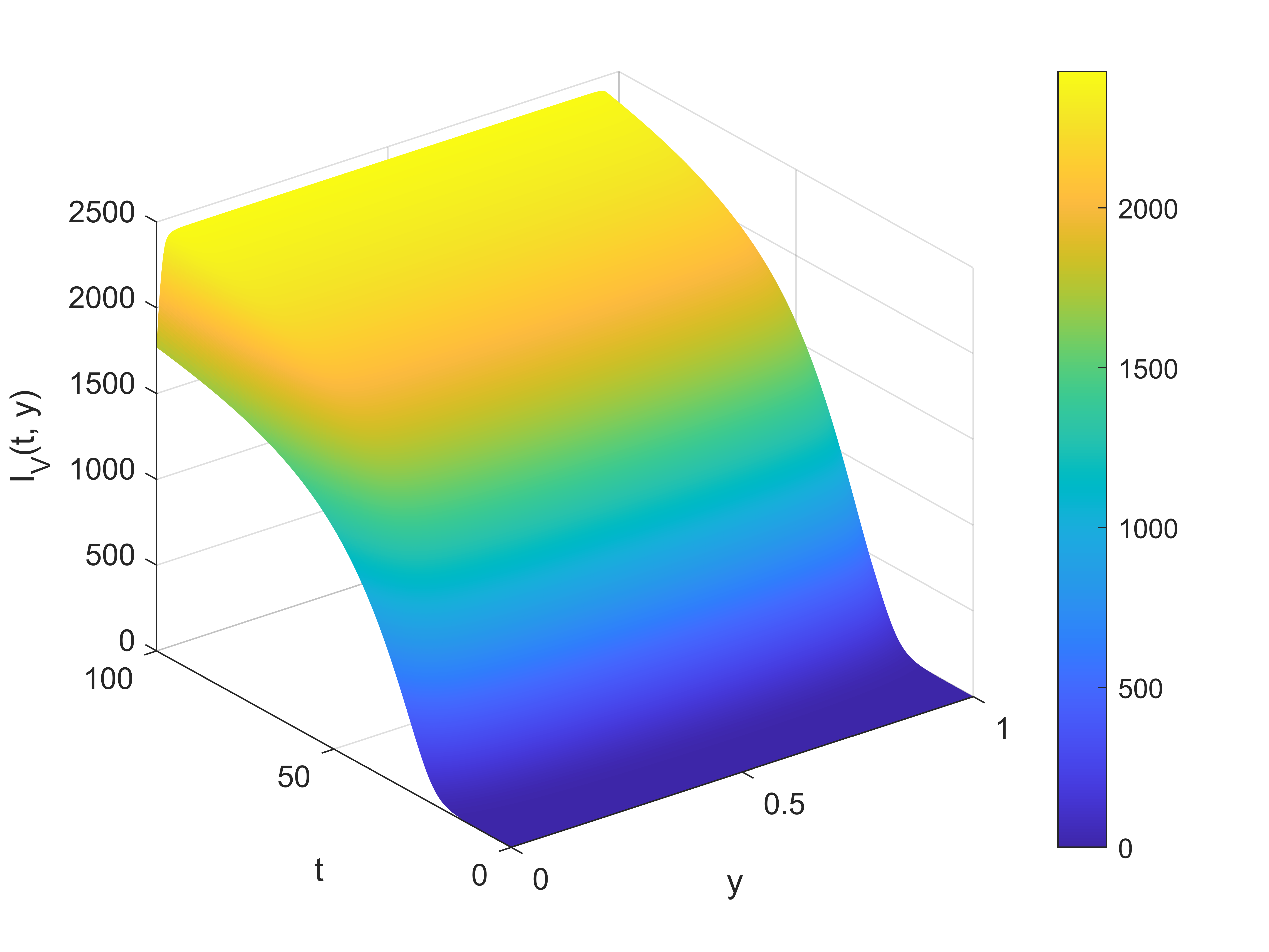}
		\end{minipage}
	}
	\hspace{0.02\textwidth}
	\subfigure[]{
		\begin{minipage}{7cm}
			\centering
			\includegraphics[scale=0.5]{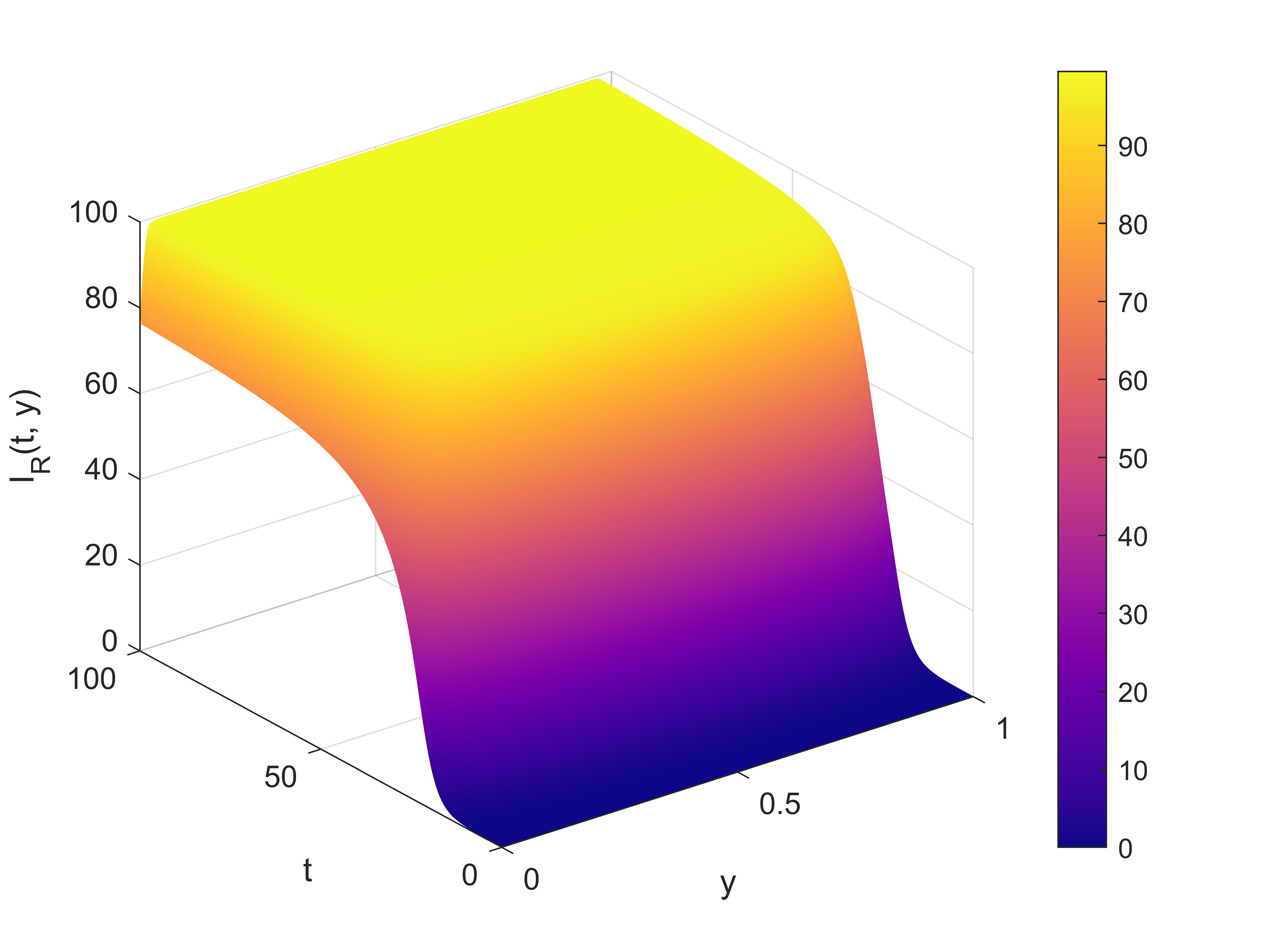}

		\end{minipage}
	}
	\caption{Dynamics of $I_V$ and $I_R $ in Case 3.}
	\label{fig:unbounded}
\end{figure}

\begin{figure}[htbp]
	\centering
	\subfigure[]{
		\begin{minipage}{7cm}
			\centering
			\includegraphics[scale=0.5]{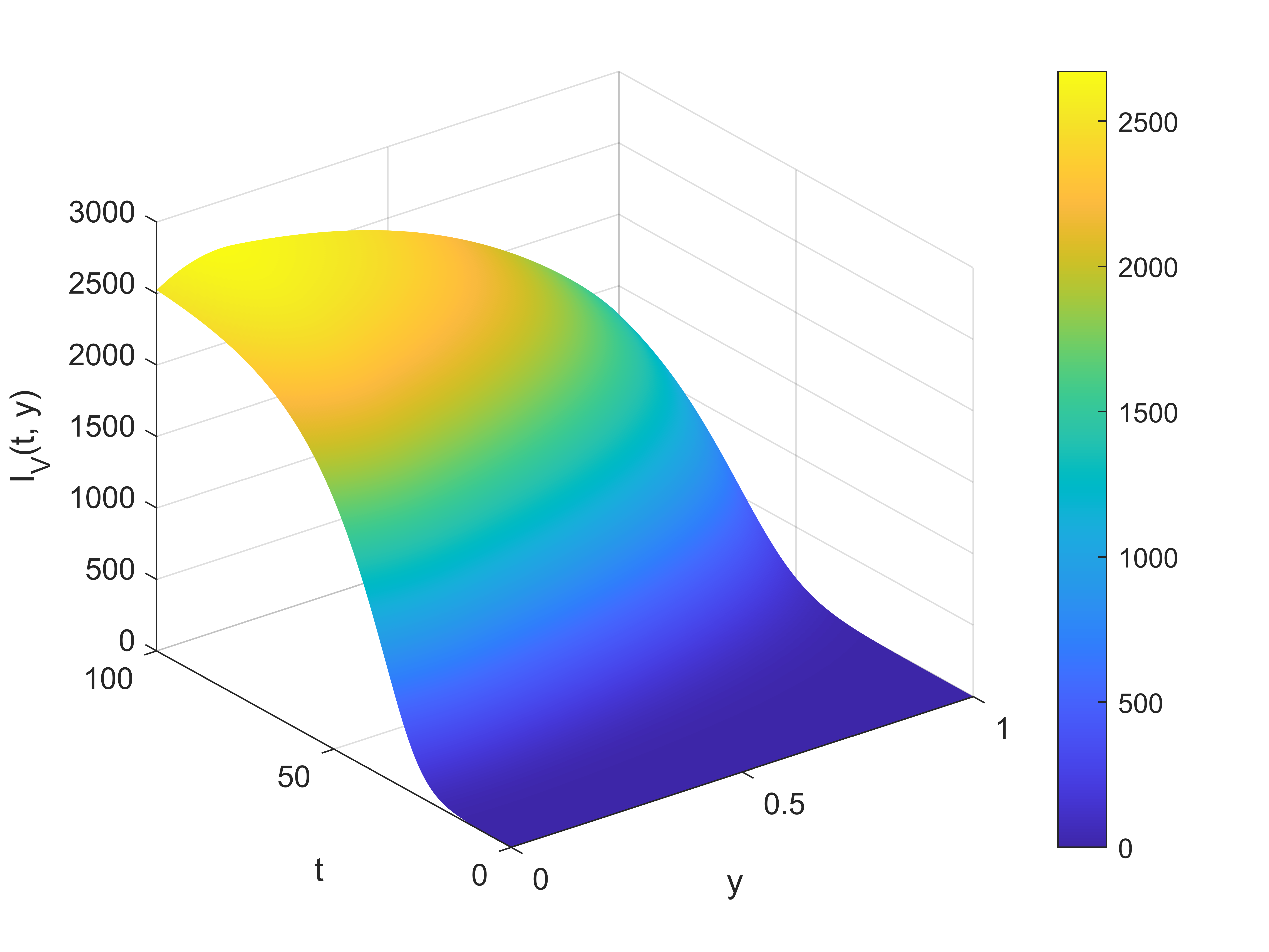}
		\end{minipage}
	}
	\hspace{0.02\textwidth}
	\subfigure[]{
		\begin{minipage}{7cm}
			\centering
			\includegraphics[scale=0.5]{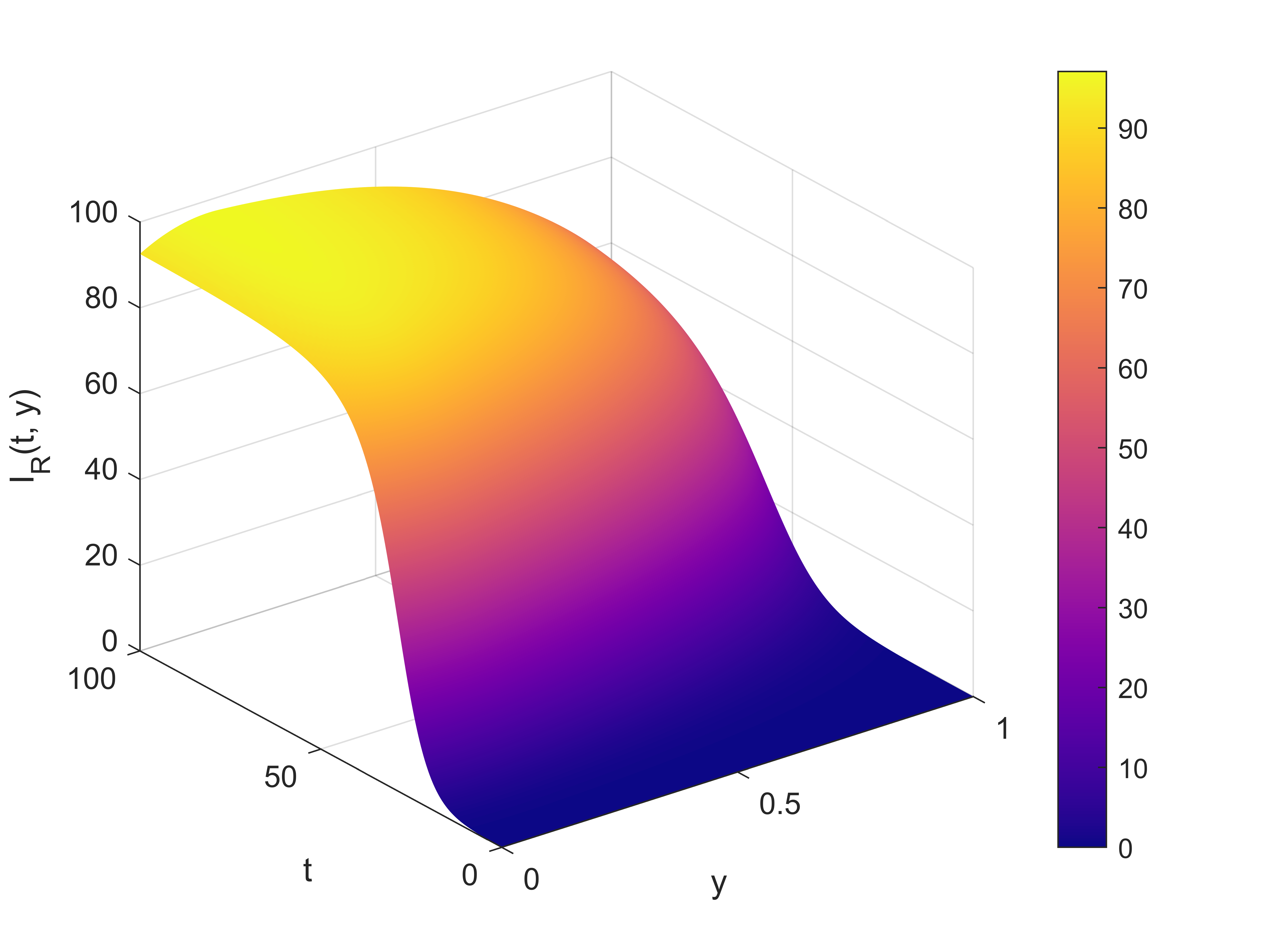}

		\end{minipage}
	}
	\caption{Dynamics of $I_V$ and $I_R $ in Case 4.}
	\label{fig:asymmetric}
\end{figure}

\newpage

\section*{Acknowledgments}
X. Lin's research was supported in part by CSC (202306380200) and NSFC (12471176).
H. Ye's research was supported in part by NSFC (12271186, 12271178), the Science and Technology Program of Shenzhen, China (20231120205244001), and  CSC (202308440378). X.-Q. Zhao's research was supported in part by the NSERC of Canada (RGPIN-2019-05648).

\end{document}